\documentclass{amsart}

\usepackage{graphicx}
\usepackage{geometry}
\usepackage{cleveref}
\usepackage{tikz-cd} 
\usepackage{amssymb}
\usepackage{amsmath}
\DeclareMathOperator{\sign}{sign}
\usepackage{amsfonts}
\usepackage{enumerate}
\usepackage{caption}
\usepackage{float}
\usepackage{wrapfig}
\usepackage{caption}
\usepackage{bbold}
\usepackage{graphicx, adjustbox}
\usepackage{subfigure}
\usepackage{epstopdf}
\usepackage{algorithmicx}
\usepackage{algpseudocode}
\ifpdf
\DeclareGraphicsExtensions{.eps,.pdf,.png,.jpg}
\else
\DeclareGraphicsExtensions{.eps}
\fi

\DeclareMathOperator{\sinc}{sinc}

\newtheorem{theorem}{Theorem}[section]
\newtheorem{lemma}[theorem]{Lemma}
\newtheorem{corollary}[theorem]{Corollary}
\newtheorem{proposition}[theorem]{Proposition}
\theoremstyle{definition}

\theoremstyle{remark}

\newtheorem{remark}[theorem]{Remark}

\numberwithin{equation}{section}

\usepackage{amsmath}
\usepackage{tikz-cd}

\usepackage{amsopn}

\newcommand{\K}{\mathcal{U}}
\newcommand{\dd}{\mathrm{d}}
\newcommand{\Proj}{\mathcal{S}}

\begin{document}

\title{On the approximation of Koopman spectra of measure-preserving flows}

\author{N. Govindarajan}
\address{Dept. of Mechanical Eng., Univerity of California at Santa Barbara, Santa Barbara, CA 93106}
\email{ngovindarajan@engineering.ucsb.edu}
\thanks{The authors gratefully acknowledge support from the Army Research Office (ARO) through grant W911NF- 11-1-0511 under the direction of program manager Dr. Samuel Stanton.}

\author{R. Mohr}
\address{Dept. of Mechanical Eng., Univerity of California at Santa Barbara, Santa Barbara, CA 93106}
\email{mohrrm@engineering.ucsb.edu}

\author{S. Chandrasekaran}
\address{Dept. of Electrical and Computer Eng., University of California at Santa Barbara, Santa Barbara, CA 93106}
\email{shiv@ece.ucsb.edu}

\author{I. Mezic}
\address{Dept. of Mechanical Eng., Univerity of California at Santa Barbara, Santa Barbara, CA 93106}
\email{mezic@engineering.ucsb.edu}


\date{\today.}


\begin{abstract}
The method of using periodic approximations to compute the spectral decomposition of the Koopman operator is generalized to the class of measure-preserving flows on compact metric spaces. It is shown that the spectral decomposition of the continuous one-parameter unitary group can be approximated from an intermediate time discretization of the flow. A sufficient condition is established between the time-discretization of the flow and the spatial discretization of the periodic approximation, so that weak convergence of spectra will occur in the limit. This condition effectively translates to the requirement that the spatial refinements must occur at a faster pace than the temporal refinements. This result is contrasted with the well-known CLF condition of finite difference schemes for advection equations.  Numerical results of spectral computations are shown for some benchmark examples of volume-preserving flows. 
\end{abstract}

\maketitle

\section{Introduction}

When it comes to the modeling of complex dynamical systems, the ability to construct models that closely follow the individual trajectories of the original system is an inherently difficult proposition. This difficulty, which arises due to exponential sensitivity of initial conditions, makes it also nearly impossible to correctly simulate long-term trajectories of a generic system whatsoever. 

Instead, a modeling philosophy in \cite{stuart1994numerical,Mezic2004} advocates precise computation of spectral objects in dynamical systems - such as invariant sets that are characterized by indicator functions in the eigenspace at $0$ of the associated Koopman operator. In many applications, it is sufficient if a reduced-order model of a system is able to correctly capture the global invariant and quasiperiodic structures that are directly detectable within the resolution of the observables. 

A mathematical formalism which takes on this dynamics of observables approach is given by Koopman operator theory \cite{Budisic2012}. The spectral decomposition of the Koopman operator provides a means of extracting low-dimensional models of observable dynamics which are capable of mimicing the relevant statistical properties. For evolutions on the attractor, the discrete parts of the Koopman spectra describe the almost periodic part of the process, while the remaining part of the process is described by the continuous spectra \cite{Mezic2005}. In a situation where discrete part of the spectrum dominates the spectral measure of an observable, a finite truncation of the spectra can provide an accurate description to the underlying process.

Paramount to the development of Koopman-based reduced order models are numerical methods that approximate the spectral decomposition of observables. In a previous paper \cite{govindarajan2018approximation}, we introduced a framework to compute the spectral decomposition of the unitary Koopman operator using the concept of ``periodic approximation'' introduced by Halmos, Katok, and Lax  \cite{Halmos1944,Katok1967,Lax1971}.  The emphasis in  \cite{govindarajan2018approximation}  was on the discrete-time case, i.e. measure-preserving maps. Here, the framework is extend to handle measure-preserving flows. In this case we have to deal with a \emph{continuous} one-parameter unitary group over the reals, instead of a discrete group over the integers.

Using the infinitesimal generator formalism, the time evolution of an observable is the solution of an advection equation. The periodic approximation we introduce is hence a numerical solution to this partial differential equation. In comparison to other methods of approximation (e.g. finite difference or semi-Lagrangian methods) \cite{bouche2003comparison,mclachlan1999area}, our proposed method does not suffer from the need to deal with instabilities or artificial damping caused by the scheme. Since the flow is directly discretized in a manner such that the measure-preserving properties are preserved, the associated Koopman operator of the discretization remains unitary, which in turn maintains the spectra on the imaginary axis, where also the spectra of the underlying infinite dimensional operator lies.. We show that our method is spectrally convergent in a weak sense. Although spectral isomorphism achieved in the limit is a weaker notion than topological conjugacy (see e.g. \cite{redei2012history} for a historical overview on the subject), it remains a prerequisite which in many cases may be sufficient for applications.

\subsection{Problem formulation} \label{sec:probdescription}
Let $S^t: X \mapsto X$ denote a Lipschitz continuous flow on a compact \emph{norm-induced} metric space $X\subseteq \mathbb{R}^m$, with $S^t$ satisfying the well-known group properties: $S^t\circ S^s(x) = S^{t+s}(x)$ for any $t,s\in \mathbb{R}$, and $S^0(x) = x$. Associate with $X$ the measure space $(X, \mathcal{M}, \mu)$, where $\mathcal{M}$ denotes the Borel sigma-algebra, and $\mu$ is an absolutely continuous measure with  full support on the state-space, i.e. $\mathrm{supp}\mbox{ }\mu =  X$. The flow $S^t$ is assumed to be invariant with respect to the measure $\mu$, i.e. for every $t\in \mathbb{R}$ and $B\in\mathcal{M}$: $\mu(B) = \mu(S^t(B))$.  

The Koopman linearization \cite{Koopman1931} of a measure-preserving flow is performed as follows. Let:
$$ L^2(X, \mathcal{M}, \mu) :=  \left\{ g:X\mapsto \mathbb{C}  \quad | \quad  \left\| g \right\| < \infty \right\}, \qquad \left\| g \right\| := \left(\int_{X} |g(x)|^2 \dd\mu(x) \right)^{\frac{1}{2}}  $$ 
denote the space of square-integrable functions on $X$ with respect to the invariant measure $\mu$. The Koopman \emph{continuous} one-parameter group $\left\{ \K^t: L^2(X, \mathcal{M}, \mu) \mapsto  L^2(X, \mathcal{M}, \mu)  \right\}_{t \in \mathbb{R}}$ is defined as the family of composition operators:
\begin{equation}
	(\K^t g)(x) :=    g \circ S^t(x), \qquad t\in\mathbb{R}. \label{eq:KOOPMAN}
\end{equation}
Since $S^t$ is an invertible measure-preserving transformation for every $t\in \mathbb{R}$, the family of operators $\{\K^t\}_{t\in \mathbb{R}}$ forms a continuous one-parameter \emph{unitary} group. In other words, \eqref{eq:KOOPMAN} is an unitary operator for every fixed $t\in \mathbb{R}$ and satisfies the group properties: $\K^t \K^s = \K^{t+s}$ for $t,s\in \mathbb{R}$, and $\K^0 = I$. As per the spectral theorem for continuous one-parameter unitary groups\cite{Akhiezer1963}, the evolution of an oberservable $g\in L^2(\mathcal{X},\mathcal{M},\mu)$ under \eqref{eq:KOOPMAN} can be decomposed as: 
\begin{equation} 
	\K^t g = \int_{\mathbb{R}} e^{i \omega t} \dd\Proj_{\omega} g, \qquad t\in \mathbb{R}.  \label{eq:decomposition}
\end{equation}
Here, $\Proj_{\omega}$ denotes a self-adjoint, \emph{projection-valued measure} on the Borel sigma-algebra $\mathcal{B}(\mathbb{R})$ of the real line $\mathbb{R}$. The {projection-valued measure} satifies the following properties:
\begin{enumerate}[(i)]
	\item For every $D \in \mathcal{B}$, 
	$$ \Proj_D := \int_{D} \dd\Proj_{\theta}$$
	is an orthogonal projector on  $L^2(X, \mathcal{M}, \mu)$.
	\item $\Proj_D = 0$ if $D=\emptyset$ and  $\Proj_D = I$ if $D= \mathbb{R}$.
	\item If $D_1, D_2 \in \mathcal{B} $ and $D_1 \cap D_2 = \emptyset$, then
	$$ \left\langle \Proj_{D_1} g   ,  \Proj_{D_2} h \right\rangle := \int_X \left(\Proj_{D_1} g\right)^*(x)  \left(\Proj_{D_2} h\right)(x)   \dd\mu(x) =  0$$
	for every $g,h\in L^2(X, \mathcal{M}, \mu)$.
	\item If $\left\{ D_k \right\}^{\infty}_{k=1}$ is a sequence of pairwise disjoint sets in $\mathcal{B}$, then
	$$ \lim_{m\rightarrow\infty} \sum_{k=1}^m  \Proj_{D_k} g =  \Proj_{D} g, \qquad D:= \bigcup_{k=1}^{\infty}  D_k   $$
	for every $g\in L^2(X, \mathcal{M}, \mu)$.
\end{enumerate}
Analogous to the discrete-time case \cite{govindarajan2018approximation}, the goal of this paper is to find an approximation to the spectral projection:
\begin{equation}
	\Proj_D g := \int_{D} \dd \Proj_{\omega} g   \label{eq:target}   
\end{equation}
for some given observable $g\in L^2(X,\mathcal{M},\mu)$ and  interval $D  \subset \mathbb{R}$. In addition, we would like to obtain an  approximating to the \emph{spectral density function}. That is, if $\varphi(\omega)\in\mathcal{D}(\mathbb{R})$ belongs to the space of smooth test functions  (i.e. Schwarz space), we wish to find $\rho(\omega;g)\in \mathcal{D}^{*}(\mathbb{R})$ in the dual space of distributions defined as the distributional derivative:
\begin{equation}
\int_{\mathbb{R}} \varphi'(\omega) c(\omega; g)  \dd\omega =  - \int_{\mathbb{R}} \varphi(\omega) \rho(\omega;g)  \dd\omega   \label{eq:spectraldensity}
\end{equation}
of the so-called spectral cumulative function on $\mathbb{R}$:
$$ c(\omega; g) :=  \langle \Proj_{(-\infty,\omega)} g , g \rangle.$$

\subsection{Main contributions} Our main contribution is the generalization of the procedure in \cite{govindarajan2018approximation} to compute the spectral decomposition of the Koopman operator for measure-preserving flows. The main technical result, which we prove here, is an asymptotic relation between the spatial and temporal discretization of the flow so that the spectra is computed correctly in the limit. This asymptotic relationship effectively states that the refinements in the spatial grid must occur at a faster rate than the temporal grid. Remarkably, the condition is, in some sense, opposite to the well-known Courant-Lewy-Friedrich (CLF) condition \cite{courant1928partiellen} which is needed for the stability of finite difference schemes.

\subsection{Related work}
Within Computional Fluid Dynamics (CFD) community, there is already some awareness on the concept of periodic approximation \cite{mclachlan1999area}. The method can be interpreted as a semi-Lagrangian method, however global efforts are made to prevent two grid points from collapsing into one.  Our work is also closely related to the development of symplectic lattice maps \cite{scovel1991symplectic}, where the goal was to preserve the structural properties of a symplectic integrator under finite machine precision through the use of integer arithmetic.

\subsection{Paper organization} Section~\ref{sec:mainresult} outlines the proposed discretization of the  Koopman operator family. Section~\ref{sec:periodicapprx} discusses the periodic approximation of measure-preserving flows, where we establish the asymptotic relation between the spatial and temporal discretization of the flow. In sections~\ref{sec:opconvergence} and \ref{sec:spectralconv}, we prove operator convergence and weak spectral convergence of the proposed scheme. Numerical examples are shown in section~\ref{sec:examples}, followed by conclusions in section~\ref{sec:conclusions}.

\section{The proposed discretization of the Koopman operator family} \label{sec:mainresult}
The discretization of the Koopman operator family can be broken-down into two steps: a temporal discretization and a spatial discretization.  

\subsection{The temporal discretization} Let $\left\{\tau(n)\right\}^{\infty}_{n=1}\subset\mathbb{R}^+$ denote a monotonically decreasing sequence converging to zero. The first step in the discretization process is to convert the flow to an automorphism by considering $\tau(n)$-map $S^{\tau(n)}: X \mapsto X$. By doing so, we obtain the discrete one-parameter group: $\{ S^{k \tau(n)} \}_{k\in \mathbb{Z}}$ along with its Koopman linearization $\left\{ \K^{k \tau(n) }: L^2(X, \mathcal{M}, \mu) \mapsto  L^2(X, \mathcal{M}, \mu)  \right\}_{k \in \mathbb{Z}}$, where:
\begin{equation} 
(\K^{k \tau(n) } g)(x) = g \circ S^{k \tau({n})}(x), \qquad   k\in\mathbb{Z}. \label{eq:dtkoopman}
\end{equation}
According to the spectral theorem \cite{Akhiezer1963}, \eqref{eq:dtkoopman} admits the decomposition:
\begin{equation}
\K^{k \tau(n) } g = \int_{\mathbb{S}} e^{i k \theta } \dd\hat{\Proj}^{\tau(n)}_{\theta} g, \qquad k\in \mathbb{Z}.   \label{eq:specdtcircle}
\end{equation}
Here, $\hat{\Proj}^{\tau(n)}_{\theta}$ is a self-adjoint, {projection-valued measure} on the Borel sigma-algebra  of the circle $\mathcal{B}(\mathbb{S})$, parameterized by $\theta\in[-\pi,\pi)$. The projection-valued measure on the circle can be mapped on the real line by introducing ${\Proj}^{\tau(n)}_{\omega}$ such that:
$$ {\Proj}^{\tau(n)}_{\omega}  = \hat{\Proj}^{\tau(n)}_{\theta}, \qquad \mbox{whenever } \theta = \left( \tau(n) \{\omega \}\right) \cap  [-\pi,\pi).$$
 By doing so, 
 \eqref{eq:specdtcircle} can be rewritten as:
\begin{equation} \K^{k \tau(n) } g = \int_{\mathbb{R}} e^{i  k \omega \tau(n) } \dd{\Proj}^{\tau(n)}_{\omega} g =  \int^{\hat{\omega}(n)}_{-\hat{\omega}(n)} e^{i  k \omega \tau(n) } \dd{\Proj}^{\tau(n)}_{\omega} g, \qquad k\in \mathbb{Z}.   \label{eq:specdtreals}
\end{equation}
where $\hat{\omega}(n)$ denotes the \emph{spectral bandwith}:
\begin{equation}
\hat{\omega}(n) = {\pi}/{\tau(n)}.
\end{equation}
For any interval $D=[a,b)\subset \left[-\hat{\omega}(n),  \hat{\omega}(n) \right)$ contained within the spectral bandwith, consider the spectral projection:
\begin{equation}
\Proj^{\tau(n)}_{D} g := \int_{D} \dd \Proj^{\tau(n)}_{\omega} g.   \label{eq:targetdt}   
\end{equation}
By comparison of \eqref{eq:specdtreals} with \eqref{eq:decomposition} and using the fact that $e^{i\theta} = e^{i\theta + 2\pi}$, the following relationship between \eqref{eq:targetdt}  and \eqref{eq:target} can be established:
\begin{equation} 
\Proj^{\tau(n)}_{D} g = \Proj_{D_n} g , \qquad D_n =   \bigcup_{l\in \mathbb{Z}}  D^{(l)}_n,\quad   D^{(l)}_n=\left[a + 2l\hat{\omega}(n) , b +  2 l \hat{\omega}(n) \right).  \label{eq:aliasing}
\end{equation}
The equality \eqref{eq:aliasing} is a consequence of \emph{aliasing}. Hence, to approximate the spectral projection \eqref{eq:target}    through the $\tau(n)$-map $S^{\tau(n)}$ will involve taking into account the errors introduced by the sets $D^{(l)}_n$ for $l\ne 0$. 

In practical calculations, this implies that $\tau(n)$ has to be chosen small enough, so that the spectral bandwith of the selected observables are suffciently captured. Consequently, observables with high frequency spectral content will require small time-steps. The conditions on the time discretization are very similar to the \emph{Nyquist-Shannon} sampling theorem, which imposes restrictions on the sample rate of a continuous-time signal so that it can be properly reconstructed.

\begin{figure}[h!]
	\begin{center}
		\includegraphics[width=.95\textwidth]{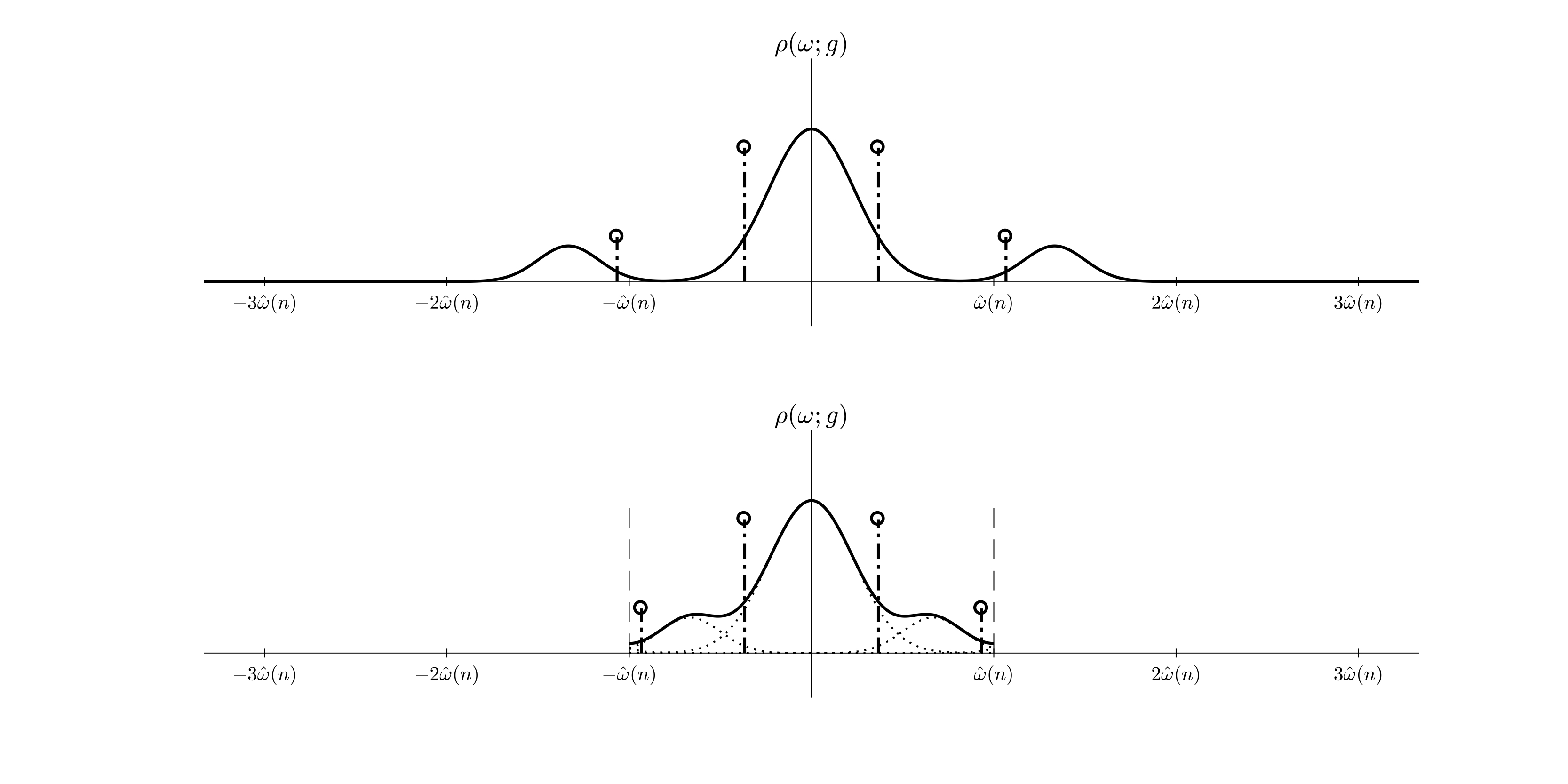} \\
	\end{center} 
	\caption{The change that the spectral density function undergoes due to aliasing \eqref{eq:aliasing}. The top figure shows the original density plot, the bottom figure shows the density plot after the time discretization.} \label{fig:aliasing}
\end{figure}

\subsection{The spatial discretization} Let $\left\{\mathcal{P}_n\right\}^{\infty}_{n=1}$ with $\mathcal{P}_n := \left\{ p_{n,1}, p_{n,2}, \ldots,  p_{n,q(n)} \right\}$ be a sequence of measurable partitions satisfying the properties:
\begin{enumerate}[(i)]
	\item Every partition element $p_{n,j}$ is compact, connected, and of equal measure, i.e.
	\begin{equation}
	\mu(p_{n,j}) = \frac{\mu(X)}{q(n)}, \qquad j\in\left\{1,2,\ldots,q(n)\right\} \label{eq:equalsized} 
	\end{equation}
	where $q:\mathbb{N} \mapsto \mathbb{N}$ is a strictly, monotonically increasing function.
	\item The diameters of the partition elements are bounded by 
	\begin{equation} \mathrm{diam}(p_{n,j}):= \sup_{x,y\in p_{n,j}} d( x, y) \leq l(n)  \label{eq:partitioncond}
	\end{equation}
	where $l:\mathbb{N} \mapsto \mathbb{R}$ is a positive, monotonic function decaying to zero in the limit.
	\item $\mathcal{P}_n$ is a refinement of $\mathcal{P}_m$ for $n>m$. That is, every $p_{m,j} \in \mathcal{P}_m$ is the union of some partition elements in $\mathcal{P}_n$.
\end{enumerate}
The second step in the discretization process is to convert the  map $S^{\tau(n)}: X \mapsto X$ on a ``continuous'' state-space to a periodic map $S^{\tau(n)}_{n}:\mathcal{P}_n \mapsto \mathcal{P}_n$ with periodicity $\zeta(n)$ on a ``discrete and finite'' state-space. The spatial discretization is performed in exactly the same manner as the original paper \cite{govindarajan2018approximation}.

At first, a discrete representation of the observable $g\in L^2(X,\mathcal{M},\mu)$ is obtained by projecting the observable $g\in L^2(X,\mathcal{M},\mu)$ onto the  finite-dimensional subspace of indicator functions,
$$  L^2_n(X, \mathcal{M}, \mu) :=  \left\{ g_n : X\mapsto \mathbb{C}  \quad | \quad  \sum_{j=1}^{q(n)} c_j  \chi_{p_{n,j}}(x), \quad c_j\in\mathbb{C}   \right\},  \qquad  \chi_{p_{n,j}}(x)  = \begin{cases} 1 & x \in p_{n,j}  \\ 0  & x \notin p_{n,j} \end{cases} $$
by means of a smoothing/averaging operation:
\begin{equation} 
(\mathcal{W}_n g)(x) = g_{n}(x) := \sum_{j=1}^{q(n)} g_{n,j}\chi_{p_{n,j}}(x), \qquad g_{n,j} = \frac{q(n)}{\mu(X)}\int_{X} g(x) \chi_{p_{n,j}}(x) \dd\mu(x). \label{eq:averageoperator}
\end{equation}
Then,  \eqref{eq:KOOPMAN} is replaced by the finite group $\left\{ \K^{k\tau(n)}_n: L^2_n(X, \mathcal{M}, \mu) \mapsto L^2_n(X, \mathcal{M}, \mu) \right\}_{k\in \mathbb{Z} / \zeta(n)}  $ defined by the permutation operators:
\begin{equation} 
\left(\K^{k\tau(n)}_n g_n\right)(x) := \sum_{j=1}^{q(n)} g_{n,j} \chi_{S^{-k\tau(n)}_n(p_{n,j})}(x), \qquad k\in \mathbb{Z} / \zeta(n).  \label{eq:discreteoperator}
\end{equation}
The eigenfunctions of $\eqref{eq:discreteoperator}$ can be expressed in terms of the basis elements of $L^2_n(X, \mathcal{M}, \mu)$, i.e. 
$$ v_{n,k} (x)= \sum_{j=1}^{q(n)} v_{n,kj}  \chi_{p_{n,j}}(x) $$
where we set  $\left\| v_{n,k} \right\| = 1$. The associated eigenvalues of $\eqref{eq:discreteoperator}$ are roots of unity: 
$$ \K^{k\tau(n)}_n v_{n,k} = e^{i\omega_{n, k} {k\tau(n)}} v_{n,k}, \qquad \omega_{n, k}:=\theta_{n, k} / \tau(n).  $$
Henceforth, the spectral decomposition of an observable $g_{n}\in L^2_n(X, \mathcal{M}, \mu)$ under the action of  \eqref{eq:discreteoperator} can be expressed as:
\begin{equation} \K^{k\tau(n)}_n g_n = \sum^{q(n)}_{k=1}  e^{ik\tau(n)\omega_{n,k}} \Proj^{\tau(n)}_{n, \omega_{n, k}} g_n \label{eq:eigdecomp},
\end{equation}
where $\Proj^{\tau(n)}_{n, \omega_{n, k}}: L^2_n(X,\mathcal{M},\mu) \mapsto  L^2_n(X,\mathcal{M},\mu)$ denotes the rank-1 self-adjoint projector:
\begin{equation}  \Proj^{\tau(n)}_{n, \omega_{n, k}} g_n =  v_{n,k} \left\langle v_{n,k}, g_n  
\right\rangle  =  v_{n,k} \left(\int_X v^*_{n,k}(x) g_n(x) \dd \mu \right) =  v_{n,k} \left(  \frac{\mu(X)}{q(n)} \sum_{j=1}^{q(n)}   v^*_{n,kj} g_{n,j}\right).
\end{equation}
Let $D=[a,b]\subset \left[-\hat{\omega}(n),  \hat{\omega}(n) \right)$. The fully discrete analogue to the spectral projection \eqref{eq:target} (both in time and space) is defined as:
\begin{equation}
\Proj^{\tau(n)}_{n, D} g_n :=  \int_{D}  \dd\Proj^{\tau(n)}_{n,\omega} g_n = \displaystyle \sum_{\omega_{n,k} \in D}  \Proj^{\tau(n)}_{n, \omega_{n, k}} g_n. \label{eq:targetapprx}
\end{equation}
Additionally, the discrete analogue of the spectral density function \eqref{eq:spectraldensity} is given by:
\begin{equation}   
\rho_n(\omega;g_n) =  \sum_{k=1}^{q(n)} \left\| \Proj^{\tau(n)}_{n, \omega_{n, k}} g_n \right\|^2 \delta(\omega- \omega_{n,k}). \label{eq:discrete}
\end{equation}
\begin{remark}
	The periodic approximation $S^{\tau(n)}_{n}: \mathcal{P}_n \mapsto \mathcal{P}_n$ already inherits the measure-preserving properties of the original flow  on the subsigma algebra generated by $\mathcal{P}_n$, i.e.  $\mu(S^{-\tau(n)} (p_{n,j})) = \mu(p_{n,j}) = \mu(S^{\tau(n)} (p_{n,j}))$. However, if the underlying flow comes from a Hamiltonian vector field, additional constraints may be imposed so that the periodic approximation also preserves the  \emph{symplectic form}.
\end{remark}
\subsection{Overview} The discretization of the Koopman operator of a measure preserving flow is split into stages. In the time-discretization, a continuous one-parameter group is replaced by a discrete one-parameter group:
$$ \left\{ \K^t: L^2(X, \mathcal{M}, \mu) \mapsto  L^2(X, \mathcal{M}, \mu)  \right\}_{t \in \mathbb{R}} \rightarrow \left\{ \K^{k \tau(n) }: L^2(X, \mathcal{M}, \mu) \mapsto  L^2(X, \mathcal{M}, \mu)  \right\}_{k \in \mathbb{Z}}. $$
In the spatial discretization, the discrete one-parameter group is replaced by a finite one-parameter group:
$$ \left\{ \K^{k \tau(n) }: L^2(X, \mathcal{M}, \mu) \mapsto  L^2(X, \mathcal{M}, \mu)  \right\}_{k \in \mathbb{Z}} \rightarrow \left\{ \K^{k\tau(n)}_n: L^2_n(X, \mathcal{M}, \mu) \mapsto L^2_n(X, \mathcal{M}, \mu) \right\}_{k\in \mathbb{Z} / \zeta(n)}.  $$
An overview of the discretization process is given in \cref{fig:discretization}.

\begin{figure}[h!]
	\begin{center}
		\includegraphics[width=.82\textwidth]{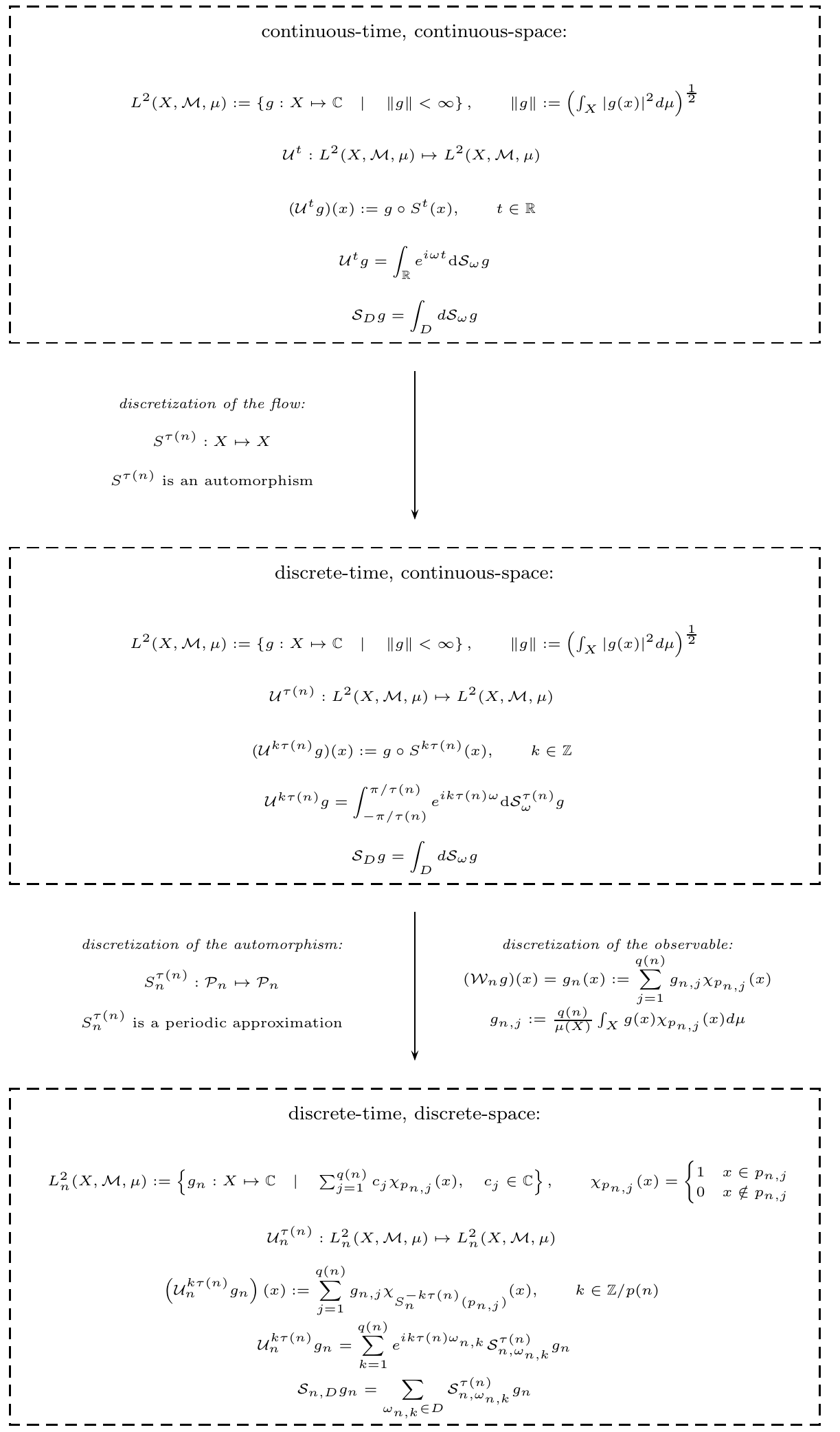} \\
	\end{center} 
	\caption{An overview of the discretization process.} \label{fig:discretization}
\end{figure}

\section{Periodic approximations of flows} \label{sec:periodicapprx}

In \cite{govindarajan2018approximation} we showed that for a measure-preserving automorphism $T:X\mapsto X$, one could construct a sequence of periodic approximations $\{T_n:\mathcal{P}_n \mapsto \mathcal{P}_n \}_{n=1}^{\infty}$ such that the dynamics $T$ is closely mimiced for longer periods of time after each consequtive refinement. More specifically, given a compact set $A\in \mathcal{M}$ and $k\in \mathbb{N}$, we showed that the set evolution of $A$ converges in the Haussdorf metric in the following sense:
\begin{equation} 
	\lim_{n \rightarrow \infty} \sum^{k}_{l = -k}  d_H(T^l(A), T^l_n(A_n))  = 0,  \label{eq:setconvergence}
\end{equation}
where:
$$ A_n := \bigcup_{p\in \mathcal{P}_n:\mbox{ } p \cap A \ne \emptyset } p \qquad \mbox{and} \qquad d_H(A,B) := \max \left\{ \sup_{a\in A} \inf_{b\in B} d(a,b),  \sup_{b\in B} \inf_{a\in A} d(a,b) \right\}.$$

An analogous statement along the lines of \eqref{eq:setconvergence} can also made for measure-preserving flows. In this section, we will show that it is possible to construct a sequence of periodic approximations $\{ S^{\tau(n)}_{n}:\mathcal{P}_n \mapsto \mathcal{P}_n \}^{\infty}_{n=1}$ to the flow  $S^t:X\mapsto X$, such that for every $t \in \mathbb{R}$:
\begin{equation}  \lim_{n\rightarrow \infty} \int^t_{-t} d_H\left( S^{s}( A ) ,  S^{\xi_{n}(s)}_{n}( A_n )   \right)  \dd s = 0,
\label{eq:flowconvergence}
\end{equation}
where:
\begin{equation}
\xi_{n}(t) = \sign (t) \left\lceil \frac{|t|}{\tau(n)}\right\rceil \tau(n). \label{stieltjes}
\end{equation}

In order for \eqref{eq:flowconvergence} to hold true, certain conditions on the spatial and temporal discretizations need to be satisfied. Recall that $S^{\tau(n)}_{n}:\mathcal{P}_n \mapsto \mathcal{P}_n$ is defined as a periodic approximation of the automorphism $S^{\tau(n)}: X \mapsto X$. But since $S^{\tau(n)}$ tends to the identity map as $n\rightarrow \infty$, it is necessary that the refinements in $\{ \mathcal{P}_n \}_{n=1}^{\infty}$ happen at a sufficiently fast rate, so that the dynamics of the flow are properly captured.

\subsection{Example: translational flow on the circle} \label{sec:example} To illustrate this technicality, we will start off with a simple example of a translation flow on the circle:
$$ S^t(x) =  (x + \Omega t)  \mod{1},\qquad \Omega\in \mathbb{R},\qquad x\in [0,1).$$
Suppose the following temporal and spatial discretizations are chosen for the flow:
$$ \tau(n) = \frac{\gamma}{\Omega w^n}, \qquad \mathcal{P}_{n} =\left\{p_{n, 1}, p_{n, 2}, \ldots, p_{n, r^n} \right\}\mbox{ with }p_{n, j} = \left(\frac{j-1}{r^n} ,\frac{j}{r^n} \right),$$
where $w,r>1$ are integers and $\gamma>0$ a positive real constant. It is not hard to derive that the mapping\footnote{$\lfloor \cdot \rceil$ denotes nearest integer function, where we choose to always round upwards for half-integers.}:
\begin{equation}
S^{\tau(n)}_{n}(p_{n, j}) =  p_{n, j^*}, \qquad j^* =  \left\lfloor  \left( j-1 + \gamma \left(\frac{r}{w}\right)^n \right) \mod r^n   \right\rceil + 1, \qquad n=1,2,\ldots,   \label{eq:example}
\end{equation}
forms a sequence of ``optimal'' periodic approximations that minimizes the cost:
$$J = \max_{p_{n, j} \in \mathcal{P}_n}d_H( S^{\tau(n)}(p_{n, j}) , S^{\tau(n)}_{n}(p_{n, j}) ).$$
The sequence of maps \eqref{eq:example} can be interpreted as an exact discretization of the flow: 
\begin{equation} \hat{S}^t_n (x) = (x + \hat{\Omega}(n) t)  \mod{1}, \qquad  \hat{\Omega}(n) = {\left\lfloor \gamma  \left(\frac{r}{w}\right)^n \right\rceil}  \frac{\Omega}{\gamma} \left(\frac{w}{r}\right)^n. 
\end{equation}
That is, 
$$d_H( \hat{S}^{\xi_n(t)}_n(p_{n, j}) , S^{\xi_n(t)}_{n}(p_{n, j}) ) = 0, \qquad t\in \mathbb{R}.$$
Whether \eqref{eq:example} satisfy the convergence criteria \eqref{eq:flowconvergence} depends on the ratio of $\tfrac{r}{w}$. Overall, we may distinguish three different situations:
\begin{itemize}
\item \emph{Case $\tfrac{r}{w}<1$.} In this scenario, the temporal discretization is refined faster than the spatial discretization. Notice that $S^{\tau(n)}_{n}$ is matched to the identity map as soon as $n\in\mathbb{N}$ is large enough to make $ \gamma \left(\tfrac{r}{w}\right)^n < \frac{1}{2}$.
Henceforth, $\displaystyle \lim_{n\rightarrow\infty }\hat{\Omega}(n) = 0$ and convergence in the sense of \eqref{eq:flowconvergence} will not occur. 
\item \emph{Case $\tfrac{r}{w}=1$.} In this scenario, the temporal discretization is refined at the same rate as the spatial discretization. The (normalized) frequency mismatch is equal to:
$$ \lim_{n\rightarrow \infty }\left| \frac{\Omega- \hat{\Omega}(n)}{\Omega}   \right|  = 1-  \frac{ {\left\lfloor \gamma   \right\rceil} }{\gamma} .$$

\item \emph{Case $\tfrac{r}{w}>1$.}	In this scenario, the spatial discretization is refined faster than the temporal discretization. In this situation,
$$ \lim_{n\rightarrow \infty }\left| \frac{\Omega- \hat{\Omega}(n)}{\Omega}   \right| = \lim_{n\rightarrow \infty } 1 - {\left\lfloor \gamma  \left(\frac{r}{w}\right)^n \right\rceil}  \frac{1}{\gamma} \left(\frac{w}{r}\right)^n   = 0.$$
Hence, convergence in the sense of \eqref{eq:flowconvergence} must occur.
\end{itemize}
Since the Koopman operator converges spectrally only when $\lim_{n\rightarrow\infty}\hat{\Omega}(n)=\hat{\Omega}$, the conditione $\tfrac{r}{w}>1$ is critical in applications of spectral computations.

\subsection{The general case: {an asymptotic requirement on the temporal and spatial discretizations}} \label{sec:flowassumptions}
The analysis of the translation flow on the circle is a specific case of a more general phenomenon that needs to be adressed in the discretization of flows. We will be using the following technical result which is a consequence of Gr{\"o}nwall inequality.

\begin{lemma} \label{lemma:gronwall}Let $S^t$ denote a lipschitz continuous flow on a compact metric space $X$. Then for some $L>0$,
	\begin{equation} d \left( S^t(x) , S^t(y) \right)\leq e^{L|t|}  d(x,y), \qquad t\in \mathbb{R} \mbox{ and }x,y\in X. 
	\end{equation}
\end{lemma}
\begin{proof}
	By compactness and Lipschitz continuity of the flow, we have that $d \left( S^{t+s}(x), S^{t}(x)\right) \leq \frac{L}{2} |s|$ and $d \left( S^{t+s}(y), S^{t}(y)\right) \leq \frac{L}{2} |s|$ for some constant $L>0$. By combining these facts, we may deduce that:
	 $$ d \left( S^{t+s}(x) , S^{t+s}(y) \right) \leq (1+L |s|)  d \left( S^t(x) , S^t(y) \right).$$
	  By repeated application of this inequality, one can show that for any $n\in\mathbb{N}$:
	$$ d \left( S^{t+s}(x) , S^{t+s}(y) \right) \leq \left(1+\frac{L |s|}{n}\right)^n  d \left( S^t(x) , S^t(y) \right).  $$
	Taking limits as $n\rightarrow \infty$, we obtain $d \left( S^{t+s}(x) , S^{t+s}(y) \right) \leq \exp(L |s|)d \left( S^t(x) , S^t(y) \right)$.
\end{proof}
We have the following theorem.
\begin{theorem} \label{thm:mainflow}
Let $S^t: X \mapsto X$ be a measure-preserving flow on a compact metric space preserving the absolutely continuous measure $\mu$ with support $\mathrm{supp}\mbox{ }\mu =  X$. Recall from \eqref{eq:partitioncond} that $\mathrm{diam}(p_{n,j}) \leq l(n)$. If:
\begin{equation} 
\lim_{n\rightarrow \infty}  \frac{ l(n)}{\tau(n)} = 0,  \label{eq:convergecondition}
\end{equation}
then there exists a sequence of periodic approximation
$\{ S^{\tau(n)}_{n}:\mathcal{P}_n \mapsto \mathcal{P}_n \}^{\infty}_{n=1}$ such that  \eqref{eq:flowconvergence} holds:
$$  \lim_{n\rightarrow \infty} \int^t_{-t} d_H\left( S^{s}( A ) ,  S^{\xi_{n}(s)}_{n}( A_n )   \right)  \dd s= 0
$$ 
for every fixed $t \in \mathbb{R}$ and compact set $A$.
\end{theorem}
The asymptotic condition \eqref{eq:convergecondition} is  \emph{sufficient} for convergence to occur, although it is not necessary as we have seen in the example shown in \cref{sec:example}. Similar to the construction in \cite{govindarajan2018approximation}, a candidate for the periodic approximation $S^{\tau(n)}_n$ is obtained from a maximum cardinality matching of a bipartite graph.
\begin{proof}[Proof of \cref{thm:mainflow}]
By the triangle inequality:	
$$   \int^t_{-t} d_H\left( S^{s}( A ) ,  S^{\xi_{n}(s)}_{n}( A_n )   \right)  \dd s \leq  \int^t_{-t} d_H\left( S^{s}( A ) ,  S^{\xi_{n}(s)}_{n}( A )   \right)  \dd s + \int^t_{-t} d_H\left( S^{\xi_{n}(s)}( A ) ,  S^{\xi_{n}(s)}_{n}( A_n )   \right)  \dd s
$$ 
Since the first term will tend to zero because of continuity of the flow $S^t$, it suffices to show that:
$$  \lim_{n\rightarrow \infty} \int^t_{-t} d_H\left( S^{\xi_{n}(s)}( A ) ,  S^{\xi_{n}(s)}_{n}( A_n )   \right)  \dd s = 0.
	$$	
The following can be further established:
\begin{eqnarray*}
	\int^t_{-t} d_H\left( S^{\xi_{n}(s)}( A ) ,  S^{\xi_{n}(s)}_{n}( A_n )   \right)  \dd s  & \leq &  \epsilon_n(t)  + \int^{t}_{-t} d_H\left( S^{\xi_{n}(s)}( A_n ) ,  S^{\xi_{n}(s)}_{n}( A_n )   \right)  \dd s     \\
	& =  &   \epsilon_n(t)  +   \int^{t}_{-t}   d_H\left( S^{\xi_{n}(s)}( A_n ) ,  S^{\xi_{n}(s)}_{n}( A_n )   \right)   \dd s \\
	& =  &   \epsilon_n(t)  +   \int^{t}_{-t}  d_H\left( \bigcup_{p\in P_n:\mbox{ }p \cap A \ne \emptyset }  S^{\xi_{n}(s)} (p)   ,  \bigcup_{p\in P_n:\mbox{ }p \cap A \ne \emptyset }  S^{\xi_{n}(s)}_{n}( p )   \right)  \dd s \\
	& \leq &  \epsilon_n(t)  +   \int^{t}_{-t} \max_{p\in P_n:\mbox{ }p \cap A \ne \emptyset }   d_H\left( S^{\xi_{n}(s)}( p) ,  S^{\xi_{n}(s)}_{n}( p )   \right)  \dd s, \\
	& \leq &  \epsilon_n(t)  +   \int^{t}_{-t} \max_{p\in P_n }   d_H\left( S^{\xi_{n}(s)}( p) ,  S^{\xi_{n}(s)}_{n}( p )   \right)  \dd s,
\end{eqnarray*}
where:
$$\epsilon_n(t) :=  \int^t_{-t} d_H\left( S^{\xi_{n}(s)}( A ) ,  S^{\xi_{n}(s)}( A_n )   \right)  \dd s \qquad\mbox{and}\qquad \lim_{n\rightarrow\infty} \epsilon_n(t) = 0. $$	
Let $\{S^{\tau(n)}_n: \mathcal{P}_n \mapsto \mathcal{P}_n \}_{n=1}^{\infty}$ denote a sequence of periodic approximations generated from a maximum cardinality matching of the bipartite graphs\footnote{According to lemma 3.2 in \cite{govindarajan2018approximation}, such a sequence will always exist given that the maximum cardinality matchings turn out to be perfect matchings.}:
\begin{equation} G_n =(\mathcal{P}_{n}, \mathcal{P}^{'}_{n} , E ), \qquad   (p_{n,k}, {p'}_{n,l}) \in E \quad\mbox{if}\quad\mu( S^{\tau(n)}(p_{n, k}) \cap p_{n, l}) > 0.  \label{eq:graphmatching} 
\end{equation} 
To establish our result, we will proceed by showing that $\{S^{\tau(n)}_n\}_{n=1}^{\infty}$ satisfies:
$$\lim_{n\rightarrow\infty} \int^{t}_{-t} \max_{p\in P_n }   d_H\left( S^{\xi_{n}(s)}( p) ,  S^{\xi_{n}(s)}_{n}( p )   \right) = 0. $$
First of all, observe that for any $p\in \mathcal{P}_n$, we have the bound:
\begin{eqnarray*} 
	d_H \left( S^{\tau(n) } (p) ,  S^{\tau(n) }_n (p) \right)  & \leq & \mathrm{diam }\mbox{ } ( S^{\tau(n)} (p)) +  \mathrm{diam }\mbox{ }(S^{\tau(n)}_n (p))  \\
	& \leq & (e^{L \tau(n)} +1 ) l(n)   \\
	& \leq & (e^{L \tau(n)} +1 ) \frac{l(n)}{\tau(n)}e^{L \tau(n)}\tau(n). 
\end{eqnarray*} 
By employing the inequality: 
\begin{eqnarray*} d_H \left( S^{ l \tau(n)} (p) ,  S^{l \tau(n)}_n (p) \right)  \leq d_H \left( S^{\tau(n) } \left(S^{ (l-1)\tau(n) } (p)  \right) ,  S^{\tau(n) } \left(S^{(l-1)\tau(n)}_n (p)  \right) \right) + \\
	 d_H \left( S^{\tau(n) } \left(S^{(l-1)\tau(n)}_n (p)  \right) ,  S^{ \tau(n)}_n \left(S^{(l-1)\tau(n)}_n (p)  \right) \right), \qquad l\in \mathbb{N}, 
\end{eqnarray*}
and using \cref{lemma:gronwall}: $d_H(S^t(A),S^t(B)) \leq e^{L|t|} d_H(A,B) $, the following recursive relation can be derived:
$$ d_H \left( S^{ l \tau(n)} (p) ,  S^{l \tau(n)}_n (p) \right) \leq e^{L \tau(n)}  d_H \left( S^{ (l-1) \tau(n)} (p) ,  S^{(l-1) \tau(n)}_n (p) \right) + (e^{L \tau(n)} +1 ) \frac{l(n)}{\tau(n)}e^{L \tau(n)}\tau(n). $$
$$    d_H \left( S^{ l \tau(n)} (p) ,  S^{l \tau(n)} (p) \right)  \leq   (e^{L \tau(n)} +1 ) \frac{l(n)}{\tau(n)} \left( e^{L \tau(n)}   +  \ldots + e^{L l\tau(n)} \right).  $$
Letting $s>0$, this may be conveniently be re-expressed into a Riemann-Stieltjes notation:
$$ d_H\left( S^{\xi_{n}(s)}( p) ,  S^{\xi_{n}(s)}_{n}( p )   \right)   (e^{L \tau(n)} +1) \frac{l(n)}{\tau(n)} \int^{\xi_{n}(s)}_{0}  e^{L | \xi_n(\sigma)|} \dd \sigma .  $$
Following an identical procedure for negative time values, we obtain a similar bound:
$$ d_H\left( S^{\xi_{n}(-s)}( p) ,  S^{\xi_{n}(-s)}_{n}( p )   \right)  \leq (e^{L \tau(n)} +1) \frac{l(n)}{\tau(n)}  \int_{\xi_{n}(-s)}^{0}  e^{L  |\xi_n(\sigma)| } \dd \sigma. $$
Combining the results yields:
$$ \int^{t}_{-t} \max_{p\in P_n }   d_H\left( S^{\xi_{n}(s)}( p) ,  S^{\xi_{n}(s)}_{n}( p )   \right)  
	 \leq  \left(e^{L \tau(n)} +1\right) \left( \frac{l(n)}{\tau(n)}\right)  \left(  \int^{t}_{-t}   \int^{\xi_{n}(s)}_{\xi_{n}(-s)} e^{L | \xi_n(\sigma)|} \dd \sigma    \dd s \right).
$$
Indeed, by taking limits:
$$ \lim_{n\rightarrow\infty} \int^{t}_{-t} \max_{p\in P_n }   d_H\left( S^{\xi_{n}(s)}( p) ,  S^{\xi_{n}(s)}_{n}( p )   \right)  
\leq  \left(2\right) \left( \lim_{n\rightarrow\infty} \frac{l(n)}{\tau(n)}\right)  \left( \frac{4(-1 + e^{L t} - L t)}{L^2} \right),
$$
and noting the condition \eqref{eq:convergecondition}, we have established what needs to be shown.
\end{proof}
\begin{remark}
In practice, $S^{\tau(n)}:X \mapsto X$ is typically not known explicitly.  Instead, one has access to an order-$s$ integrator $\tilde{S}^{\tau(n)}:X \mapsto X$ which acts as an approximator to the flow. Assuming that the integrator preserves the invariant measure of the flow, the question arises whether replacing $S^{\tau(n)}_n:\mathcal{P}_n \mapsto \mathcal{P}_n$ with  $\tilde{S}^{\tau(n)}_n:\mathcal{P}_n \mapsto \mathcal{P}_n$ would still allow \cref{thm:mainflow} to fall through. The answer to this question is affirmative. This follows from the fact that $d(S^{\tau(n)}(x), \tilde{S}^{\tau(n)}(x)) = \mathcal{O}({\tau^{s+1}(n)})$ and the triangle inequality: 
$$
d_H (S^{\tau(n)} (p), \tilde{S}^{\tau(n)}_n (p) ) \leq  d_H (S^{\tau(n)} (p), \tilde{S}^{\tau(n)} (p) ) + d_H (\tilde{S}^{\tau(n)} (p) , \tilde{S}^{\tau(n)}_n (p) ). 
$$
\end{remark}
\begin{remark}
Notice that the condition on the spatial and temporal discretizations \eqref{eq:convergecondition} is different than, and in some sense opposite the Courant-Friedrichs-Lewy (CFL) condition which typically arises in finite difference schemes of Hyperbolic PDEs. In \cref{sec:advection}, we revisit this matter in greater detail.
\end{remark}

\section{Operator convergence} \label{sec:opconvergence}

In this section we establish operator convergence for a convergencing sequence of periodic approximations to a flow. Since the proofs are very similar to the discrete-time case \cite{govindarajan2018approximation}, most details of the proof are left out.
\begin{lemma} \label{lemm:indicator}
Suppose that $\left\{S^{\tau(n)}_{n}: \mathcal{P}_{n} \mapsto \mathcal{P}_{n} \right\}^{\infty}_{n=1}$ is a sequence of discrete maps that periodically approximates $S^t:X\mapsto X$  in the sense of  \cref{thm:mainflow}.   For some $m\in\mathbb{N}$, define: 
	$$g = \sum_{j=1}^{q(m)} c_{j} \chi_{p_{m,j}} \in  L^2_m(X,\mathcal{M},\mu). $$
	Then, for any fixed $t\in \mathbb{R}$:
$$ \lim_{n\rightarrow \infty}   \int^t_{-t} \left\| \K^{s} g - \K^{\xi_{n}(s)}_n g_n \right\|^2  \dd s = 0.$$	
\end{lemma}
\begin{proof}
	Again, by continuity it suffices to just show:
	$$ \lim_{n\rightarrow \infty}   \int^t_{-t} \left\| \K^{\xi_{n}(s)} g - \K^{\xi_{n}(s)}_n g_n \right\|^2  \dd s = 0.$$
	For notational clarity, write $A^{(j)} := p_{m,j}$, $g^{(j)} := \chi_{p_{m,j}}$ and $ g^{(j)}_n := \mathcal{W}_n g^{(j)}$. We have:
	\begin{eqnarray*}
 \int^t_{-t} \left\| \K^{\xi_{n}(s)} g - \K^{\xi_{n}(s)}_n g_n \right\|^2  \dd s
		& \leq & q(m) \left( \sum_{j=1}^{q(m)}  |c_{j}|^2 \right)   \max_{j=1,\ldots, q(m)} \left(   \int^t_{-t} \left\| \K^{\xi_{n}(s)} g^{(j)} - \K^{\xi_{n}(s)}_n g^{(j)}_n \right\|^2  \dd s \right).    
	\end{eqnarray*}
	Since $\left\{\mathcal{P}_n\right\}^\infty_{n=1}$ are consecutive refinements, observe that $g^{(j)}_n = g^{(j)} $ for $n\geq m$, which implies: 
	\begin{eqnarray*}    \int^t_{-t} \left\| \K^{\xi_{n}(s)} g^{(j)} - \K^{\xi_{n}(s)}_n g^{(j)}_n \right\|^2  \dd s  =   \int^t_{-t} \mu \left( S^{\xi_n(s)}(A^{(j)}) \Delta S^{\xi_n(s)}_n(A^{(j)})  \dd \xi_{n}(s) \right)\dd s, \qquad \mbox{if } n\geq m.
	\end{eqnarray*}
One may use lemma 4.1 from \cite{govindarajan2018approximation} in combination with \cref{thm:mainflow} to establish the desired result. 
\end{proof}

\begin{theorem}[Operator convergence] \label{thm:opconvergence}
Suppose that $\left\{S^{\tau(n)}_{n}: \mathcal{P}_{n} \mapsto \mathcal{P}_{n} \right\}^{\infty}_{n=1}$ is a sequence of discrete maps that periodically approximates $S^t:X\mapsto X$  in the sense of  \cref{thm:mainflow}. Then, for every fixed  $t\in \mathbb{R}$:
	\begin{equation} 	\lim_{n \rightarrow \infty}   \int^t_{-t} \left\| \K^{s} g - \K^{\xi_{n}(s)}_n g_n \right\|^2 \dd s = 0, \label{eq:opconvergence}
	\end{equation} 
	where $g\in L^2(X,\mathcal{M},\mu)$ and $\xi_n(s)$ is defined by \eqref{stieltjes}.
\end{theorem}
\begin{proof}
	The proof is very similar to discrete case \cite{govindarajan2018approximation}. One can pick a 
	$$g_m = \sum_{j=1}^{q(m)} c_{j} \chi_{p_{m,j}} \in  L^2_m(X,\mathcal{M},\mu) $$
	that approximates $g\in L^2(X,\mathcal{M},\mu)$ arbitrarily well in the norm-wise sense by choosing a sufficiently large $m\in\mathbb{N}$. 
\end{proof}

\section{Spectral convergence} \label{sec:spectralconv}

In this section, we establish results related to spectral convergence. In particular, we will examine how \eqref{eq:targetapprx} converges to \eqref{eq:target}.

\subsection{Approximation of the spectral projectors} 

Consider any smooth test function $\varphi\in\mathcal{D}(\mathbb{R})$ on the reals, and define:
$$ \Proj_{\varphi} g = \int_{\mathbb{R}} \varphi (\omega) \dd\Proj_{\omega} g,\qquad  \Proj^{\tau(n)}_{n, \varphi} g_n =  \displaystyle \sum_{\omega_{n,k} \in \mathbb{R} } \varphi (\omega_{n, k}) \Proj^{\tau(n)}_{n, \omega_{n, k}} g_n.$$
We will prove the following. 

\begin{theorem} \label{thm:projectorconvergence}
	Suppose that $\left\{S^{\tau(n)}_{n}: \mathcal{P}_{n} \mapsto \mathcal{P}_{n} \right\}^{\infty}_{n=1}$ is a sequence of discrete maps that periodically approximates $S^t:X\mapsto X$  in the sense of  \cref{thm:mainflow}.  For any smooth test function $\varphi\in\mathcal{D}(\mathbb{R})$ and observable $g \in L^2(X,\mathcal{M},\mu)$, we have:
	$$  \lim_{n\rightarrow \infty }\left\| \Proj_{\varphi} g - \Proj^{\tau(n)}_{n, \varphi} g_n \right\|   = 0.$$
\end{theorem}
\begin{proof}
	Express the test function in terms of its Fourier transform:
	$ \varphi (\omega) = \int_{-\infty}^{\infty}  b(\tau) e^{i\tau\omega} \dd\tau, $
	and note that $\int_{-\infty}^{\infty}  | b(\tau)| \dd\tau < \infty$. We see that:
	$$ \Proj_{\varphi} g =\int_{\mathbb{R}} \left( \int_{-\infty}^{\infty}  b(\tau) e^{i\tau\omega} \dd\tau \right) \dd\Proj_{\omega} g \\
	= \int_{-\infty}^{\infty} b(\tau) \left( \int_{\mathbb{R}} e^{i \tau \omega}  \dd\Proj_{\omega} g  \right) \dd \tau\\
	=  \int_{-\infty}^{\infty} b(\tau)  \K^\tau g \dd \tau, $$
	where we employed the spectral theorem of unitary one-parameter groups \cite{Akhiezer1963} in the last equality. Similarly, it also holds that:
	$$  \Proj^{\tau(n)}_{n, \varphi} g_n =   \int_{-\infty}^{\infty} b(\tau)   \K^{\xi(\tau)}_n g_n \dd \tau $$
	Hence,
	$$ \Proj_{\varphi} g -  \Proj^{\tau(n)}_{n, \varphi} g_n = \int_{-\infty}^{\infty} b(\tau) (  \K^{\tau} g -  \K^{\xi(\tau)}_n g_n) \dd \tau $$
	Now let $\epsilon> 0$ and choose $t\in \mathbb{R}$ such that:
	$$ \int_{|\tau|> t } | b (\tau) | \dd \tau < \frac{\epsilon}{ 4 \left\| g \right\| } $$
	Noting that:
	\begin{eqnarray*}
		\left\| \Proj_{\varphi} g -  \Proj_{n, \varphi} g_n \right\| 
		& \leq &  M \int_{-t}^{t}   \left\|  \K^{\tau} g -  \K^{\xi(\tau)}_n g_n  \right\| \dd \tau  + \frac{\epsilon}{2},   \\
	\end{eqnarray*}
	and using \cref{thm:opconvergence} the proof can be completed.
\end{proof}

Just as in the discrete-time case \cite{govindarajan2018approximation}, smoothen the indicator function $\chi_D(\omega)$ using the summability kernel, i.e.
$$\chi_{D_{\alpha}} (\omega) = \int_{\mathbb{S}} \varphi_{\alpha}(\theta,\xi)  \chi_D(\xi) \dd \xi,$$ 
where $\varphi_{\alpha}: \mathbb{R} \times \mathbb{R} \mapsto \mathbb{R}_+$:
\begin{equation}\varphi_{\alpha}(x,y) = \begin{cases}   \frac{K}{\alpha} \exp\left({\frac{-1}{1- \left(\frac{|x-y|}{\alpha} \right)^2 }}\right) &  \frac{|x-y|}{\alpha} < 1    \\
0 & \mbox{otherwise}                          
\end{cases}, \label{eq:kernel}
\end{equation} 
for some $\alpha>0$ and $K=(\int^{1}_{-1} \exp(\tfrac{-1}{1-x^2})  \dd x)^{-1}$. Define:
\begin{equation}
\Proj_{D_{\alpha}}g = \int_{\mathbb{R}} \chi_{D_{\alpha}} (\omega) \dd\Proj_{\omega} g,\qquad  \Proj^{\tau(n)}_{n, {D_{\alpha}}} g_n =  \displaystyle \sum_{\omega_{n,k} \in \mathbb{R} } \chi_{D_{\alpha}} (\omega_{n, k}) \Proj_{n, \omega_{n, k}} g_n.
\end{equation}
We have the following corollary.

\begin{corollary}[Convergence of spectral projectors] \label{corr:projectorconvergence}
	Suppose that $\left\{S^{\tau(n)}_{n}: \mathcal{P}_{n} \mapsto \mathcal{P}_{n} \right\}^{\infty}_{n=1}$ is a sequence of discrete maps that periodically approximates $S^t:X\mapsto X$  in the sense of  \cref{thm:mainflow}. Given any $\alpha>0$ and interval $D \subset \mathbb{R}$, it follows that:
	$$  \lim_{n\rightarrow \infty }\left\| \Proj_{D_{\alpha}} g - \Proj^{\tau(n)}_{n, {D_{\alpha}}} g_n \right\|   = 0, $$
	where $g \in L^2(X,\mathcal{M},\mu)$. 
\end{corollary}

\subsection{Approximation of the spectral density function}
Recall the definition of the spectral density function along with its discrete analogue in \eqref{eq:discrete}. To assess the convergence of $\rho_n(\omega;g_n)$ to $\rho(\omega;g)$, we again make use of summability kernels \eqref{eq:kernel}.
\begin{theorem}[Approximation of the spectral density function]
	Let:
	$$ \rho_{\alpha}(\omega; g) :=  \int^{\infty}_{-\infty} \varphi_{\alpha}(\omega,\xi)  \rho(\xi; g)  \dd \xi, \qquad \rho_{\alpha,n}(\omega; g_n) := \int^{\infty}_{-\infty}  \varphi_{\alpha}(\omega,\xi)  \rho_{n}(\xi; g_n)  \dd \xi.  $$
	It follows that:
	$$\lim_{n\rightarrow\infty} \rho_{\alpha,n}(\omega; g_n)  =  \rho_{\alpha}(\omega; g), \quad \mbox{uniformly}.
	$$
\end{theorem}
\begin{proof} Similar to \cite{govindarajan2018approximation}, to prove uniform convergence, we will establish that: (i) $\rho_{\alpha}(\omega; g) - \rho_{\alpha,n}(\omega;g_n)$ forms an equicontinuous family, and (ii) $\rho_{\alpha,n}(\omega;g_n)$ converges to $\rho_{\alpha}(\omega;g)$ in the $L^2$-norm.
	\begin{enumerate}[(i)]
		\item To show that $\rho_{\alpha}(\omega; g) - \rho_{\alpha,n}(\omega;g_n)$ is an equicontinuous family, we will confirm that its derivative $\rho^{'}_{\alpha}(\omega; g) - \rho^{'}_{\alpha,n}(\omega;g_n)$ is uniformly bounded. Write:
		$$ \rho_{\alpha}(\omega; g) - \rho_{\alpha,n}(\omega;g_n) = \frac{1}{2\pi}\int^{\infty}_{-\infty} b_n(\tau;g) e^{i \tau \omega}\dd \tau ,\qquad b_n(\tau;g) :=   \int^{\infty}_{-\infty} e^{-i\tau\omega}  \left( \rho_{\alpha}(\omega;g)  -   \rho_{\alpha,n}(\omega;g_n) \right)  \dd \omega.  $$
		According to the spectral theorem of unitary operators \cite{Akhiezer1963}, we have by construction that:
		$$ a(\tau;g) := \int_{-\infty}^{\infty} e^{-i\tau\omega} \rho(\omega;g) \dd \theta =  \langle  g , \K^{\tau} g \rangle,\quad a_n(\tau;g_n) := \int_{-\infty}^{\infty} e^{-i\tau \omega} \rho_n(\omega;g_n) \dd \omega =  \langle  g_n , \K^{\xi_{n}(\tau)}_n g_n \rangle, \quad \tau\in \mathbb{R}.$$
		The functions $\rho_{\alpha}(\omega; g)$ and $\rho_{\alpha,n}(\omega;g_n)$ are defined as convolutions with a function belonging to the \emph{Schwartz space}. Recognizing that convolutions implies pointwise multiplication in Fourier domain, we obtain:
		$$   b_n(\tau;g) =  d_{\alpha}(\tau)   ( a(\tau;g) - a_n(\tau;g_n) ), $$
		where:
		$$ d_{\alpha}(\tau) :=  \int_{-\infty}^{\infty} e^{-i\tau\omega}  \varphi_{\alpha}(\omega,0) \dd \omega \quad \mbox{and} \quad |d_{\alpha}(\tau)|\leq \frac{C_{\alpha}}{1+|\tau|^N} \mbox{ for every }N\in\mathbb{N}. $$ 
		Now examining the  derivative $\rho^{'}_{\alpha}(\omega; g) - \rho^{'}_{\alpha,n}(\omega;g_n)$ more closely, we see that:
		\begin{eqnarray*}
			\left| \rho^{'}_{\alpha}(\omega; g) - \rho^{'}_{\alpha,n}(\omega;g_n) \right| & = &  \left| \frac{1}{2\pi} \int^{\infty}_{-\infty} it b_n(\tau;g) e^{i \tau \omega} \dd \tau \right| \\
			& \leq &   \frac{1}{2\pi}  \int^{\infty}_{-\infty} |\tau| |d_{\alpha}(\tau)|\left| \langle  g , \K^{\tau} g \rangle - \langle  g_n , \K^{\xi_n(\tau)}_n g_n \rangle \right| \dd \tau \\
			&  \leq & \frac{ \left\| g \right\|^2}{\pi} \int^{\infty}_{-\infty}  \frac{C_{\alpha} |\tau| }{1+|\tau|^N} \dd \tau,
		\end{eqnarray*}
		which is a convergent sum for $N\geq 3$.
		\item To show that $\rho_{\alpha,n}(\omega;g_n)$ converges to $\rho_{\alpha}(\omega;g)$ in the $L^2$-norm, we will use Parseval's identity to confirm that the integral:
			$ \int^{\infty}_{-\infty} \left| b_n(\tau;g)\right|^2 \dd \tau $
		can be made arbitrarily small. At first, note that:
		$$ a(\tau;g) - a_n(\tau;g_n) =   \langle g ,   \K^{t} g -  \K^{\xi_n(\tau)}_n g_n \rangle   -   \langle g - g_n   ,   \K^{\xi_n(\tau)}_n g_n \rangle $$
		By the triangle inequality and Cauchy-Schwarz, we obtain:
		\begin{eqnarray*}  |a(\tau;g) - a_n(\tau;g_n)| 	& \leq &   \left\| g \right\| \left(  \left\| \K^{\tau} g -  \K^{\xi_n(\tau)}_n g_n  \right\|    +   \left\| g - g_n \right\|          \right).
		\end{eqnarray*}
		Let $\epsilon>0$, and choose $t\in \mathbb{R}_+$ such that:
		\begin{equation} 
		\int_{|\tau|>t}   |d_{\alpha}(\tau)|^2 \dd \tau \leq \epsilon.   \label{eq:yeahbaby}
		\end{equation}
		This is always possible, because $\varphi_{\alpha}(\theta,0)$ is a $C^{\infty}$ smooth function, and therefore also square-integrable. The following upper bound can be established:
		\begin{eqnarray*}
			\int^{\infty}_{-\infty} \left| b_n(\tau;g)\right|^2 \tau
			& = & 	\int^{\infty}_{-\infty} |d_{\alpha}(\tau)|^2 |a(\tau;g) - a_n(\tau;g_n)|^2 \dd \tau \\
			& \leq &  \int^{\infty}_{-\infty}   |d_{\alpha}(\tau)|^2  \left( \left\| g \right\| \left( \left\| \K^{\tau} g -  \K^{\xi_n(\tau)}_n g_n \right\|    +   \left\| g - g_n \right\| \right)    \right)^2 \dd\tau  \\
			& \leq & \left\| g \right\|^2  \max_{-t \leq \tau \leq t} |d_{\alpha}(\tau)|^2    \int^{t}_{-t} \left\| \K^{\tau} g -  \K^{\xi_n(\tau)}_n g_n \right\|^2 \dd\tau + 16\left\| g \right\|^2   \int_{|\tau|>t}   |d_{\alpha}(l)|^2 \dd \tau \\
			&   & + \left\| g \right\|^2  \left\| g - g_n \right\| \int^{t}_{-t}   |d_{\alpha}(\tau)|^2  \left(  2\left\| \K^{\tau} g -  \K^{\xi_n(\tau)}_n g_n \right\| +    \left\| g - g_n \right\| \right) \dd \tau.
		\end{eqnarray*}
		Now apply \eqref{eq:yeahbaby} and \cref{thm:opconvergence} to complete the proof.
	\end{enumerate}
\end{proof}

\section{Some remarks on the simulation of advection equations}  \label{sec:advection}
The generator of the associated Koopman unitary group is the operator $f(x) \cdot \nabla$, 
where $f(x)$ is the vector field that generates the measure-preserving flow $S^t(x)$ . Subsequently, the time evolution of an observable under the flow $S^t(x)$ is equivalent to the solution of an advection equation associated with the vector field $f(x)$. 
 Henceforth, the discretization \eqref{eq:discreteoperator} is an approximate solution propagator to the advection problem:
\begin{subequations}
\begin{eqnarray}
\left(\left(\frac{\partial}{\partial t} - \mathcal{G}  \right) \phi \right) (t,x) &=& 0\\
 \left(\mathcal{A} \phi\right)(t,x)  & = &\phi_0(x),  \label{eq:boundarycond}
\end{eqnarray} \label{eq:advection}
\end{subequations}
where:
$$ \mathcal{G}\phi(t,x)  := f(x) \cdot \nabla \phi(t,x) ,\qquad   \left(\mathcal{A} \phi\right)(t,x) := \left. \phi(t,x)\right|_{t=0}. $$
Within the Computional Fluid Dynamics (CFD) community, there is already some familiarity on the concept of periodic approximation. Specifically, McLachan \cite{mclachlan1999area} coined the term ``cell rearrangement model'' to describe such approximation schemes. In the previous sections, we have shown that these methods are convergent both in a spectral sense and operator sense. That is, if $\phi_n(t,x) := \K^{\xi_{n}(t)}_n \mathcal{W}_n \phi_0(x)$, then for any fixed $t\in \mathbb{R}$:
$$
	(i).\quad \lim_{n\rightarrow \infty }\left\| \phi(t,x) - \phi_n(t,x)  \right\|  = 0,\qquad
	(ii).\quad \lim_{n\rightarrow \infty }\left\| \Proj_{\varphi} \phi(t,x) - \Proj^{\tau(n)}_{n,\varphi} \phi_n(t,x)  \right\|  = 0. 
$$

The advection equation is a hyperbolic PDE which has been heavily studied in the literature. A whole plethora of alternative numerical schemes can be used to solve  system \eqref{eq:advection}. A simple finite-difference scheme can already exhbibit very different convergence properties than the periodic approximation approach. Consider the \emph{translation flow} on the circle which was examined in \cref{sec:example}. The generator $\mathcal{G}$ in this case equals $\Omega \tfrac{\partial}{\partial x}$ and is spatially invariant. Recall the temporal and spatial discretizations which were used in the periodic approximation:
$$\l(n)=\frac{1}{r^n},\quad \tau(n) = \frac{\gamma}{\Omega w^n}.  $$
A \emph{first-order upwind} finite-difference scheme yields a sequence of discretizations of the Koopman operators, whose matrix representation is of the type:
 \begin{equation} \hat{U}^{\xi_n(t)}_n =  \begin{bmatrix} 1- \gamma \left(\frac{r}{w}\right)^n &   &  &  \gamma \left(\frac{r}{w}\right)^n \\
  \gamma \left(\frac{r}{w}\right)^n  &   1- \gamma \left(\frac{r}{w}\right)^n &   \\
 &           \ddots                     & \ddots  &\\
                   &                                &   \gamma \left(\frac{r}{w}\right)^n     & 1- \gamma \left(\frac{r}{w}\right)^n                    \end{bmatrix}^{\frac{\xi_n(t)}{\tau(n)}}.   \label{eq:upwind}
 \end{equation}
 For a periodic approximation, weak convergence of the spectra was guaranteed when $\tfrac{r}{w}>1$ (see again \cref{sec:example}). For the upwind scheme discretization, this is no longer true:
\begin{itemize}
	\item \emph{Case $\tfrac{r}{w}<1$.}  In this situation,  \eqref{eq:upwind} converges to the identity map. However, unlike the periodic approximation, the limit is never achieved and only holds true in the asymptotic sense. Weak convergence in the spectra is nevertheless not achieved.
	\item \emph{Case $\tfrac{r}{w}=1$.} In this situation, \eqref{eq:upwind} is reduced to: \\
	\begin{minipage}{0.5\textwidth}
		\begin{equation} \hat{U}^{\xi_n(t)}_n =  \begin{bmatrix} 1- \gamma &   &  &  \gamma  \\
	\gamma  &   1- \gamma &   \\
	&           \ddots                     & \ddots  &\\
	&                                &   \gamma    & 1- \gamma                    \end{bmatrix}^{\frac{\xi_n(t)}{\tau(n)}}. 
	\label{eq:ulamfinitedifference}
	\end{equation}
	In the case when $0 <\gamma \leq 1$,   \eqref{eq:ulamfinitedifference} is the time evolution of  a doubly-stochastic, circulant matrix.
	 In fact, the operator can be interpreted as an Ulam approximation (see e.g. \cite{dellnitz1999approximation}) of the $\tau(n)$-map for the underlying flow:  
	 $$ T_n(x) = \left( x + \frac{1}{{\gamma} w^n} \Omega \right) \mod 1.$$
	 As illustrated in  \cref{fig:sample_figure}, the parameter $\gamma\in(0,1]$ signifies the probability of jumping to next partition interval. For the special $\gamma=1$, the operator is equivalent to a periodic approximation. 
	 Nevertheless, the eigenvalue-eigenfunction pairs of the discretized operator \eqref{eq:ulamfinitedifference}:
	   $$ \left( \hat{\K}^{\xi_n(t)}_n  v_{n,j}\right) (x) = e^{\lambda_{j,n} \xi_n(t) } v_{n,j} (x)$$
	 for $j=1,\ldots,r^n$ can be found explicitly,
	  \end{minipage}
	  \begin{minipage}{0.45\textwidth}
	  	\begin{center}
	  	\includegraphics[width=.8\textwidth, trim={2cm 2cm 0cm 2cm}, clip]{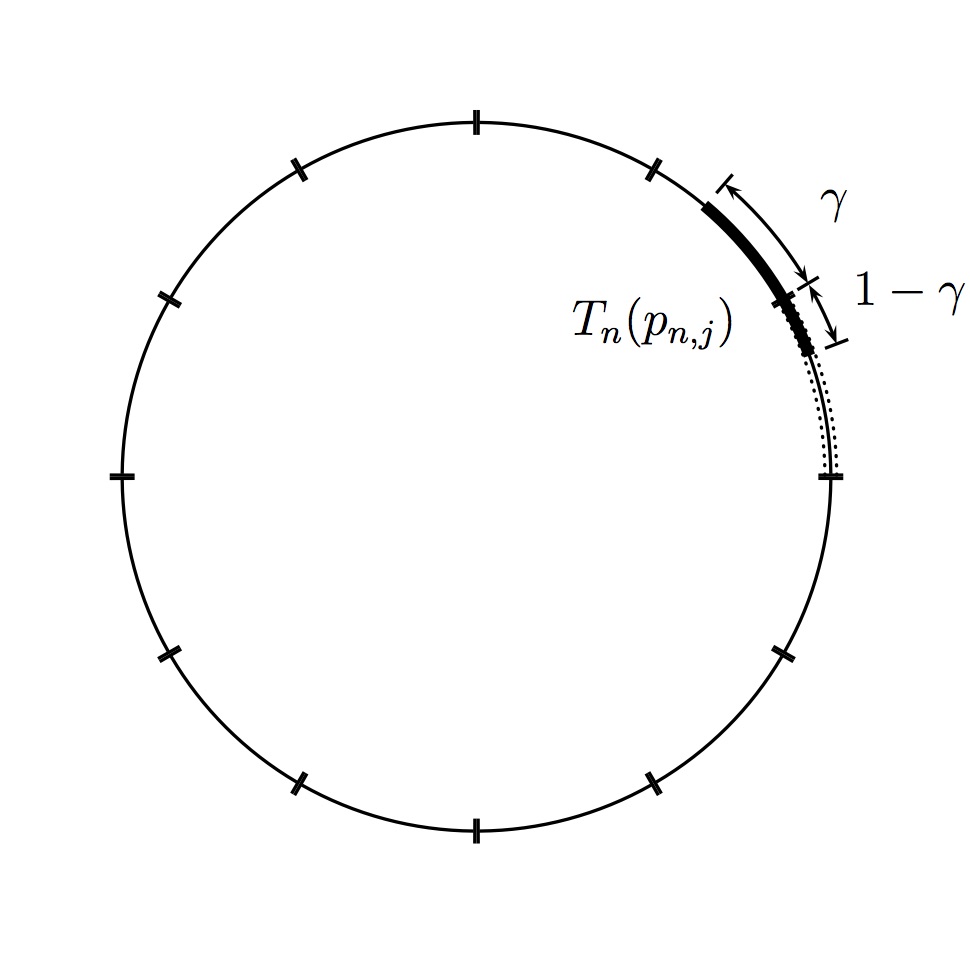}
	  	\captionof{figure}{The ``Ulam interpretation'' to the discretized operator \eqref{eq:ulamfinitedifference}.}
	  	\label{fig:sample_figure}
	  	
	  \end{center}
	  \end{minipage}
	  
	$$  \lambda_{n,j} = \left(2 \pi \kappa_n(j) \Omega\right) i  + \frac{\Omega r^n}{\gamma} \log \left(  \frac{ 1-\gamma + \gamma e^{\frac{ 2 \pi \kappa_n(j)}{r^n}  i}  }{  e^{\gamma\frac{ 2 \pi \kappa_n(j) }{r^n}  i}   }\right),\quad v_{n,j} (x)= \sum_{k=1}^{r^n} e^{\frac{2\pi \kappa_n(j)(k-1)}{r^n}   i } \chi_{p_{n,j}}(x),   $$
	where:
	$$\kappa_n(j) = \left( j-1-\frac{r^n}{2}\right)\mod r^n  -\frac{r^n}{2}.  $$
	
	In \cref{fig:ulamspectra} the location of eigenvalues are plotted for varying $\gamma>0$ and $n\in\mathbb{N}$. The eigenvalues only remain on the imaginary axis in the special case $\gamma=1$ when  \eqref{eq:ulamfinitedifference} reduces  to  a periodic approximation. 
	\begin{figure}[h!]
		\subfigure[Eigenvalues for $n=7$ and varying $\gamma>0$.]{\includegraphics[width=.37\textwidth]{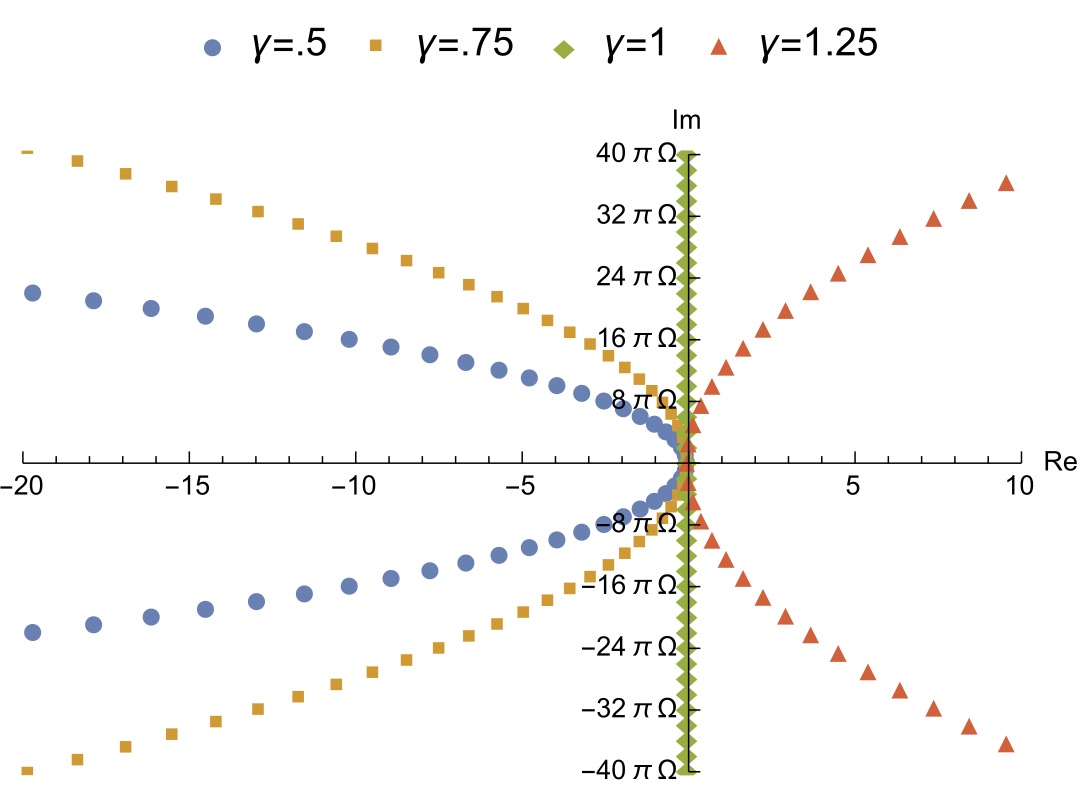}}
		\subfigure[Eigenvalues for $\gamma=0.75$ and varying $n\in\mathbb{N}$.]{\includegraphics[width=.37\textwidth]{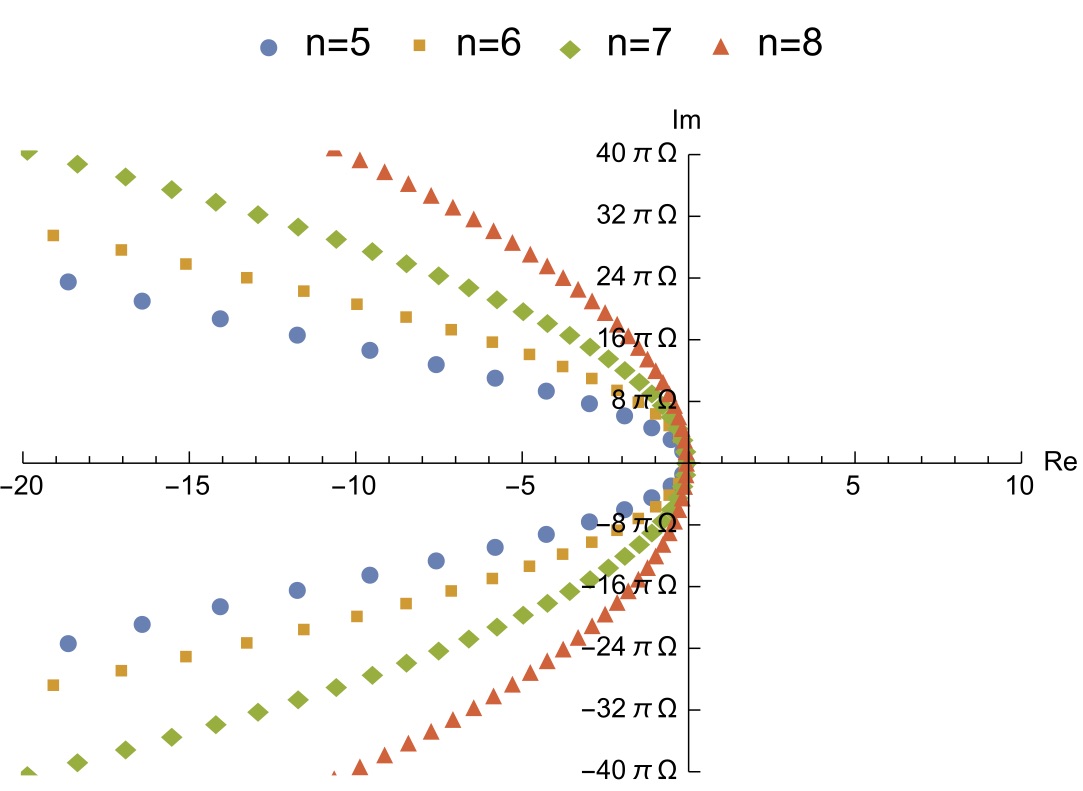}  }
		\caption{Eigenvalues of \eqref{eq:ulamfinitedifference} for $r=w=2$ and varying $\gamma>0$ and $n\in\mathbb{N}$. } \label{fig:ulamspectra}
	\end{figure}
	If $0 <\gamma \leq 1$ (the CLF condition for the upwind scheme), the eigenvalues deflect off to the left-half plane, whereas for $\gamma>1$ they deflect off to right-half plane. Although refinements on the partition do push eigenvalues corresponding to slow modes closer to the imaginary axis, the eigenvalues corresponding to fast modes always either get dissipated or amplified.
	
	\item \emph{Case $\tfrac{r}{w}>1$.}	 The entries in \eqref{eq:upwind} grow unboundedly. Hence, it is impossible to have any operator or weak spectral convergence in this scenario.
\end{itemize}

In general, statements on the spectral convergence properties of a finite diffference schemes are harder to make. The generator is typically not spatially invariant and there is no longer a specific structure to exploit other than the sparsity. The uniqueness of the periodic approximation is that the discretization preserves the unitary structure of the underlying operator. In \cite{mclachlan1999area}, McLachan indicated that periodic approximations surpress the formation of spurious oscillations in simulations. The restriction of the spectra to the imaginary axis may play a critical role here, since it prevents instabilities and artificial damping.

\section{Numerical examples} \label{sec:examples}

For symplectic flows, and volume preserving flows in general, algorithms for constructing periodic approximations are relatively straightforward. Periodic approximations can be obtained in either two ways: directly through brute fore using bipartite matching algorithms (see e.g. \cite{kloeden1997constructing,govindarajan2018approximation}), or indirectly through developing a symplectic/volume-preserving integrator that leaves a lattice invariant (i.e. symplectic lattice maps, see \cite{scovel1991symplectic}). In this section, we will examine the Koopman spectral properties of some low dimensional (i.e. $\leq 3$) Hamiltonian and volume-preserving flows.

\subsection{Hamiltonian systems}

Hamiltonian systems are defined on unbounded domains. But in many situations, the trajectories belonging to a sub-level set of energy surfaces are bounded. Hence, for that subdomain of the state-space, periodic approximations may be constructed to compute spectra. For a separable, one degree-of-freedom Hamiltonian systems, periodic approximations can be obtained readily from a symplectic lattice map (see \cite{scovel1991symplectic}). We will examine the spectra of the simple pendulum and the duffing oscillator.  

\subsubsection{Simple pendulum} Consider the simple pendulum:
\begin{equation}
\begin{bmatrix} \dot{x}_1 \\  \dot{x}_2 \end{bmatrix} = \begin{bmatrix} x_2 \\ -\sin x_1 \end{bmatrix}, \label{eq:pendulum}
\end{equation}
and let us restrict ourselves to the domain:
$$ X=\left\{ x_1\in[0,2\pi), x_2 \in \mathbb{R}:\quad   \tfrac{1}{2} x^2_2 - \cos(x_1) \leq \tfrac{1}{2}\pi^2 +1 \right\}.$$
Apart from the single eigenvalue at $\lambda=0$, it was shown in \cite{Mezic2017} that the spectra of \eqref{eq:pendulum} is fully continuous. In \cref{fig:densitypend}, we plot the spectra of the observable:
\begin{equation}
g(x_1,x_2) = \frac{1}{2} x^2_2 - \left( \cos x_1 \right). \label{eq:obspend}
\end{equation}
In \cref{fig:projpend}, spectral projections are shown for various intervals.

\subsubsection{Duffing oscillator} Consider the duffing oscillator:
\begin{equation}
\begin{bmatrix} \dot{x}_1 \\  \dot{x}_2 \end{bmatrix} = \begin{bmatrix} x_2 \\  - b x_1 - a x^3_1  \end{bmatrix}, \label{eq:duffing}
\end{equation}
restricted  to the domain:
$$ X=\left\{ x_1\in\mathbb{R}, x_2 \in \mathbb{R}:\quad   \tfrac{1}{2} x^2_2 + \tfrac{1}{2} b x^2_1 + \tfrac{1}{4} a x^4_1 \leq \tfrac{1}{2} \pi^2+ \tfrac{1}{2} b \pi^2 + \tfrac{1}{4} a \pi^4 \right\}.$$
Set $b=-1$, $a=1$ and consider the observable: 
\begin{equation} 
g(x_1,x_2) = \frac{1}{2} x^2_2 - \frac{1}{2}  x^2_1 + \frac{1}{4}  x^4_1  \label{eq:obsduffing1}
\end{equation}
In \cref{fig:densityduffing} the spectra is plotted and in \cref{fig:projduffing} projections are shown.

Let us fix $a = 1$ and vary the coefficient $b$ from negative to positive values. 
In \cref{fig:spectdensityduffing1,fig:spectdensityduffing2}, we plot the spectral density of the observable:
\begin{equation} 
g(x_1,x_2) = \frac{1}{2} x^2_2 + i \left( \frac{1}{2}  x^2_1  \right).  \label{eq:obsduffing}
\end{equation}
When $b=0$, the system undergoes a pitch-fork bifurcation. In \cref{fig:spectdensityduffing1}, the spectral density is plotted for small perturbations of $b$, i.e. around the bifurcation point. From a topological point of view, the duffing oscillator clearly undergoes a sudden transition with the birth of two new fixed points. But from a spectral sense, this transition is however smooth and unnoticable. The smooth transitation can be clarifed by the close proximity of the new fixed points during the bifurcation. Noticable changes in the spectra occur only when $b$ is modified signifanctly, as evident in  \cref{fig:spectdensityduffing2}. Since the bifurcation does not induce immediate global topological changes, from a spectral point of view, the transition will remain smooth as the Koopman framework inherently incorporates finite resolution in measurement and observation.

\subsection{The quadruple gyre}
Next, we consider a variation to the double gyre dynamics introduced in \cite{shadden2005definition}. The quadruple gyre dynamics are described by the differential equations:
\begin{equation}
\begin{bmatrix} \dot{x}_1 \\ \dot{x}_2 \\ \dot{x}_3 
\end{bmatrix} = \begin{bmatrix}  \pi A \sin(\pi f_1(x_1,{x}_3)) \cos (\pi f_2(x_2,{x}_3)) \frac{\dd f_1(x_1,{x}_3)}{\dd x_1} \\  -\pi A \cos(\pi f_1(x_1,{x}_3)) \sin (\pi f_2(x_2,{x}_3)) \frac{\dd f_2(x_2,{x}_3)}{\dd x_2} \\ 1 \end{bmatrix} \label{eq:quadruplegyre}
\end{equation}
on the domain $X=[0,1]\times [0,1] \times [0,1)$, with:
$$  f_1(x_1,{x}_3) = 4\epsilon \sin (2\pi x_3)  {x^2_1} + (2-4 \epsilon  \sin (2\pi x_3) ) x_1, \qquad f_2(x_2,{x}_3) = 4\epsilon \sin (2\pi x_3)  {x^2_2} + (2-4 \epsilon  \sin (2\pi x_3) ) x_2.$$
The system \eqref{eq:quadruplegyre} arises from a time-periodic stream function. In fact, the variable $x_3$ is periodic and equal to the time (modulo the period). 

In \cref{fig:spectdensitygyre}, the spectral density function is plotted for the observable:
\begin{equation}
g(x_1,x_2,x_3) = i ( \sin(4\pi x_1) \sin(4\pi x_2)  )  + 4 \psi(x_1, x_2), \label{eq:obsgyre}
\end{equation}
where:
$$ \psi(x_1, x_2) = \begin{cases} \exp\left(\frac{-1}{1- \tfrac{1}{2}\sqrt{(x_1-\tfrac{1}{2})^2 + (x_2-\tfrac{1}{2})^2}  }\right) & (x_1-\tfrac{1}{2})^2 + (x_2-\tfrac{1}{2})^2 \leq \tfrac{1}{4}  \\ 0 & (x_1-\tfrac{1}{2})^2 + (x_2-\tfrac{1}{2})^2 \leq \tfrac{1}{4}  \end{cases}.$$
In our calculations, we set $A=1/(2\pi)$ and $\epsilon = 0.05$. We used a spatial partition of $700\times 700 \times 100$ and the time step was set to $\tau = 0.01$. As evident from \cref{fig:spectdensitygyre}, it appears that the quadruple gyre has a mixed spectrum for these parameters.  The location of the discrete spectra correspond to resonant frequencies of the KAM tori islands shown in  \cref{fig:poincare}. This is also noticable in \cref{fig:gyre1,fig:gyre2,fig:gyre3}, where  spectral projections are shown for certain intervals of interest.

\subsection{The Arnold-Beltrami-Childress flow}
Finally, we consider the Arnold-Beltrami-Childress (ABC) flow on the unit 3-torus, i.e. $X=[0,1)^3$. The motion is described by the differential equations:
\begin{equation}
\begin{bmatrix} \dot{x}_1 \\ \dot{x}_2 \\ \dot{x}_3 
\end{bmatrix} = \begin{bmatrix}  A \sin x_3   + C \cos x_2 \\   B \sin x_1   + A \cos x_3  \\ C \sin x_2 + B \cos x_1 \end{bmatrix}. \label{eq:ABC}
\end{equation}
For small $\tau$-values, we may approximate the flow $S^{\tau}$ of \eqref{eq:ABC}  by:
\begin{equation} 
\tilde{S}^{\tau}(x) =  \tilde{S}^{\tau/3}_1 \circ \tilde{S}^{\tau/3}_2 \circ \tilde{S}^{\tau/3}_3 (x) \label{eq:ABCmap}
\end{equation}
where:
$$ \tilde{S}^{\tau/3}_1 (x) = \begin{bmatrix} x_1 + \frac{\tau}{3} ( A \sin x_3   + C \cos x_2 ) \\ x_2 \\ x_3 \end{bmatrix},\quad \tilde{S}^{\tau/3}_1 (x) = \begin{bmatrix} x_1 \\ x_2 + \frac{\tau}{3} ( B \sin x_1   + A \cos x_3 )\\ x_3 \end{bmatrix},$$
$$\tilde{S}^{\tau/3}_1 (x) = \begin{bmatrix} x_1 \\ x_2 \\ x_3 +  \frac{\tau}{3}(C \sin x_2 + B \cos x_1 )\end{bmatrix}.  $$
This approximate volume-preserving $\tau$-map is called the ABC map \cite{Feingold1988} and is the composition of three shear maps. A periodic approximation of \eqref{eq:ABCmap} is obtained readily by periodically approximating each of the shear maps separately. In \cref{fig:spectdensityABC}, we plot the spectral density function for the observable:
\begin{equation}
g(x_1,x_2,x_3) =   \exp(4\pi i x_2) + 2 \exp(6 \pi i x_1) +  \exp(2 \pi i x_3)  \label{eq:obsABC}
\end{equation}
For our calculations, we set $A = \sqrt{3}/(2\pi)$, $B = \sqrt{2}/(2\pi)$, $C = 1/(2\pi)$. We used a spatial partition of $400\times 400 \times 400$ and the time step was set to $\tau = 0.025$.  In \cref{fig:projABC1,fig:projABC2}, spectral projection are shown onto various intervals.

\section{Conclusions} \label{sec:conclusions}
We generalized the concept of periodic approximations to measure-preserving flows. The additional time discretization which had to be dealt with required special care. An asymptotic condition was established between the time-discretization of the flow and the spatial discretization of the periodic approximation, so that weak convergence of the spectra will occur in the limit. Effectively, this condition necessitates that the spatial refinements must occur at a faster rate than the temporal refinements. It is interesting to note that this requirement is opposite to what the CLF condition dictates for stability of finite difference schemes.

\bibliographystyle{amsplain}
\bibliography{references}

\begin{figure}[b!]
	\includegraphics[width=.7\textwidth]{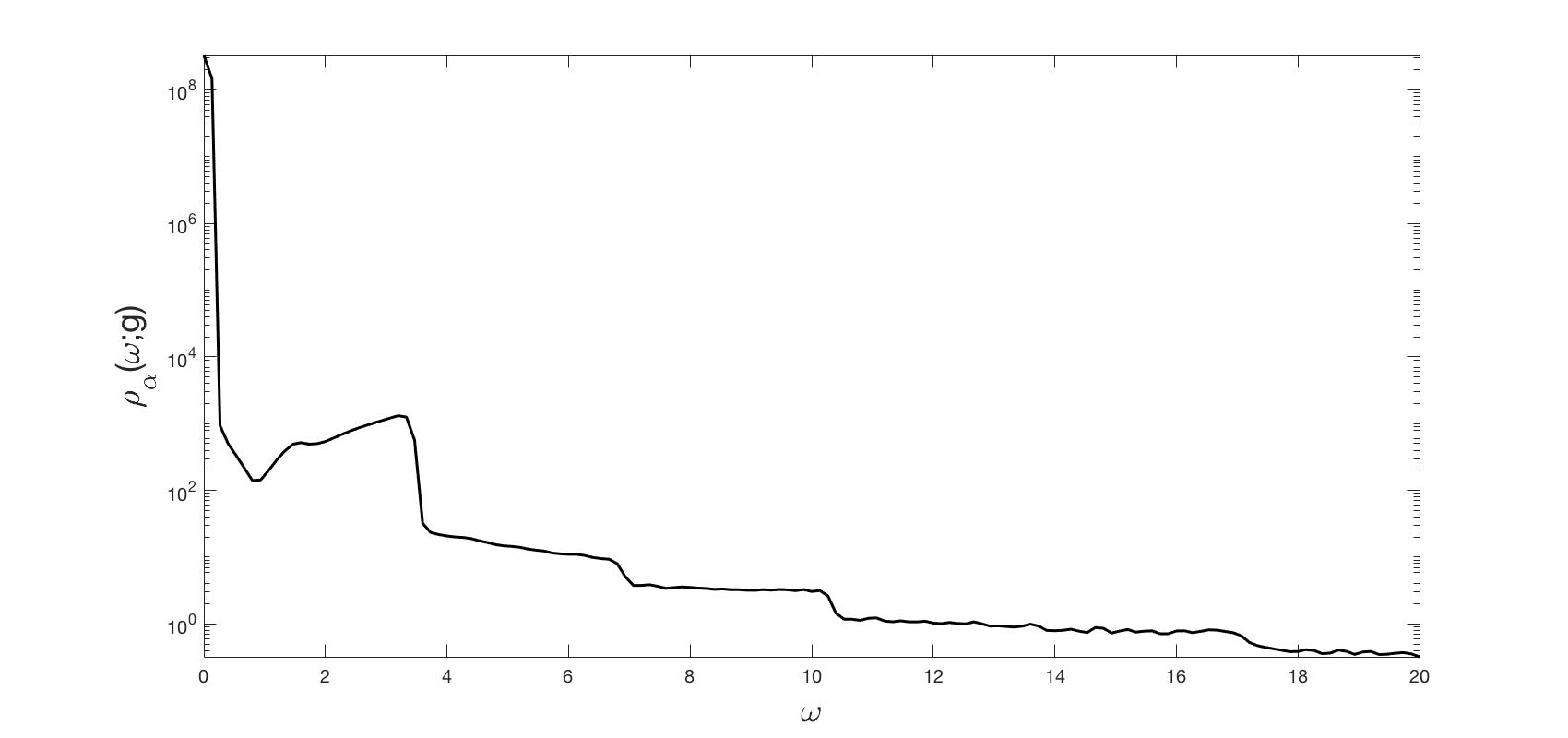}
	\caption{The spectral density function ($\alpha=0.1$) for the observable \eqref{eq:obspend}. } \label{fig:densitypend}
	\begin{center}
		\begin{tabular}{cc}
			\includegraphics[width=.47\textwidth]{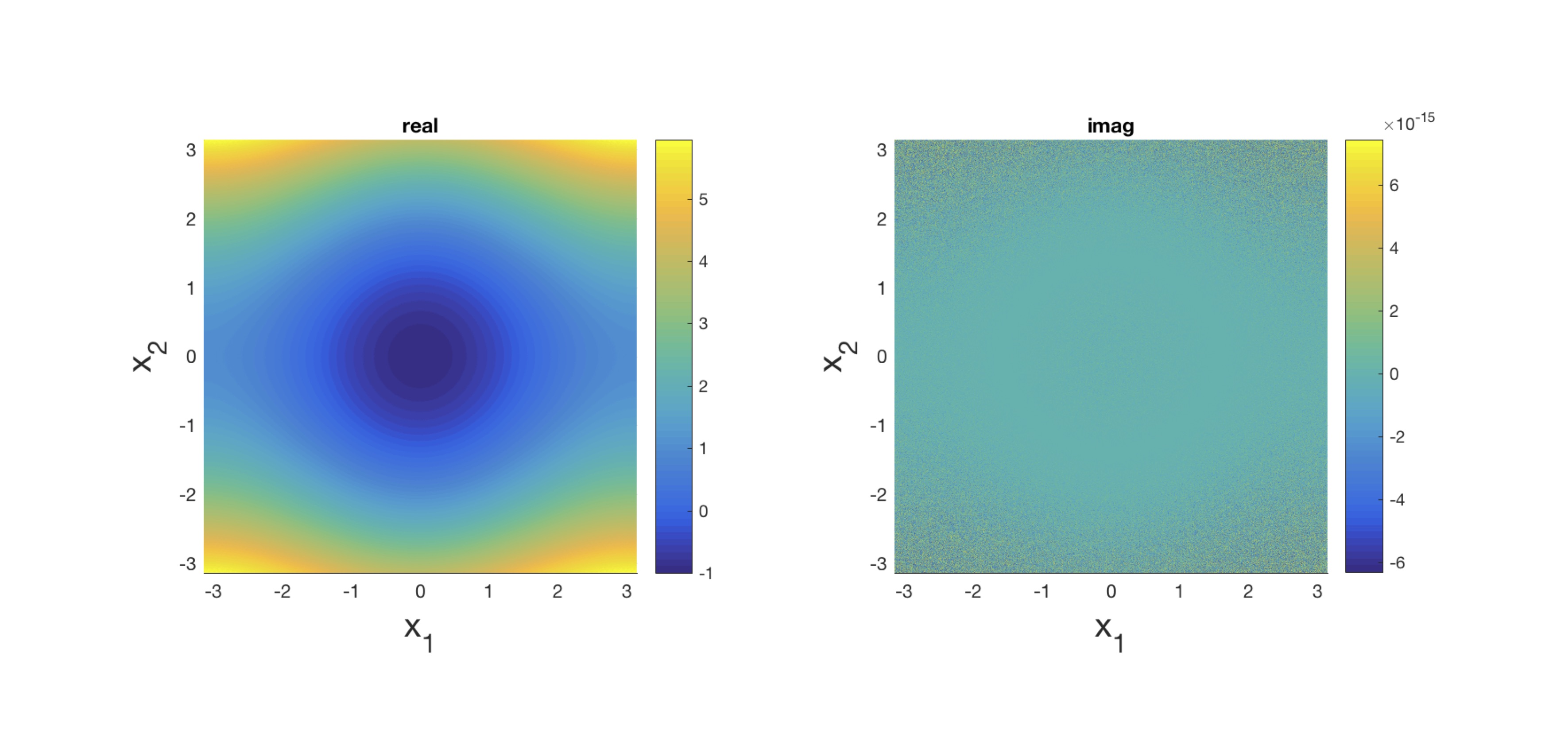} &  \includegraphics[width=.47\textwidth]{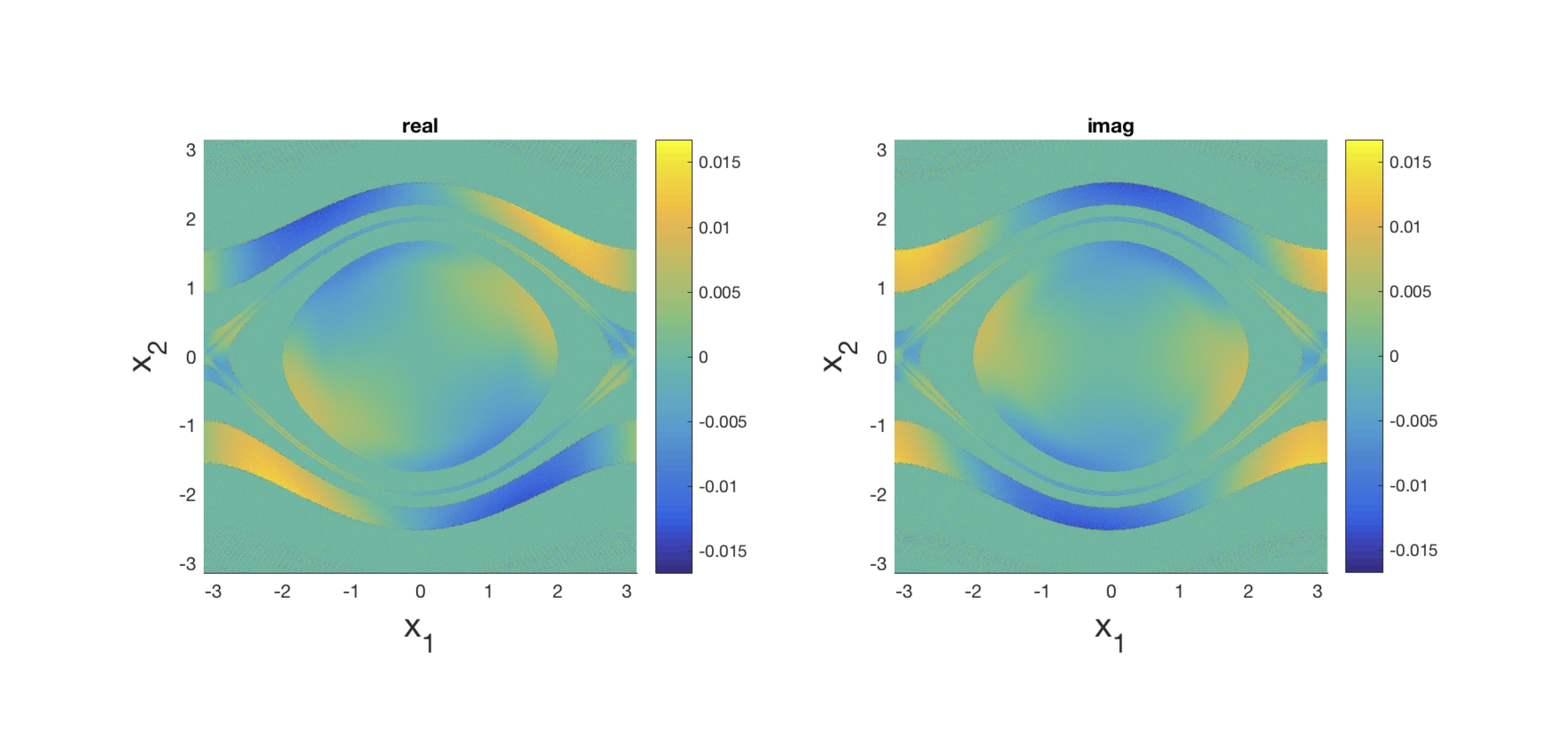} \\
			\includegraphics[width=.47\textwidth]{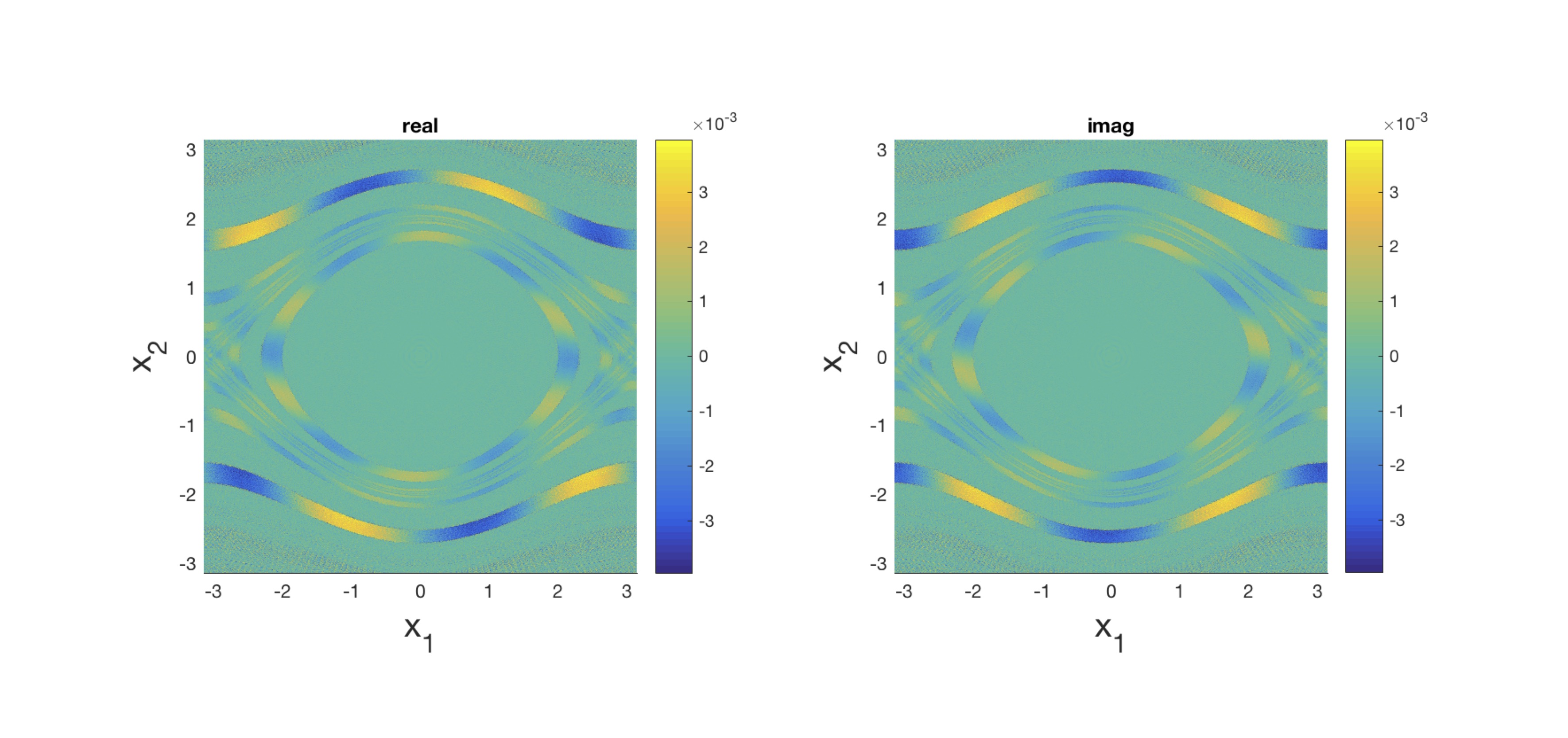} & 
			\includegraphics[width=.47\textwidth]{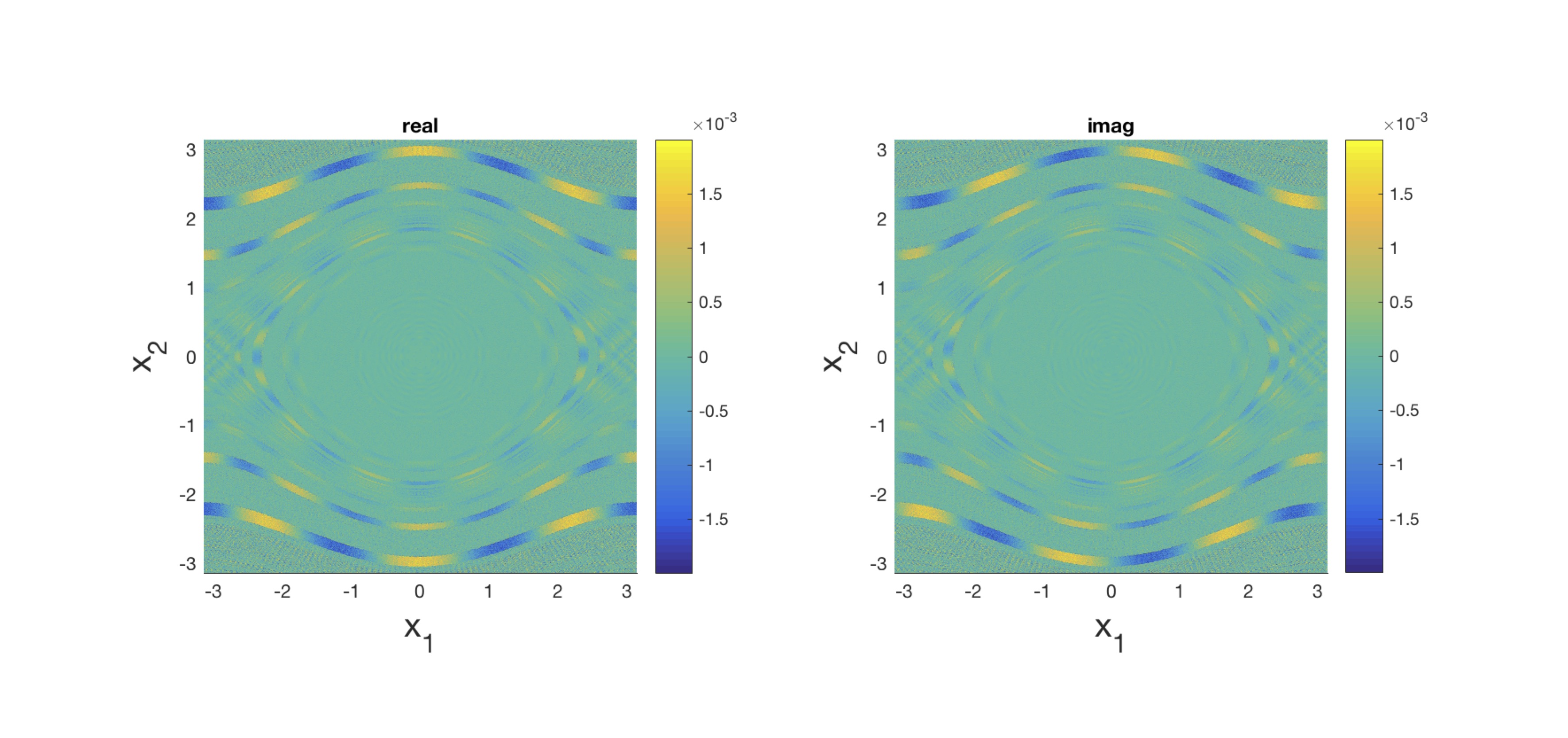} 
		\end{tabular}
	\end{center}
	\caption{Spectral projections for the observable \eqref{eq:obspend} on the invervals $D=[-0.3,0.3)$ (top-left), $D=[1.5,2.0)$ (top-right), $D=[4.0,4.5)$ (bottom-left), and $D=[7.5,8.0)$ (bottom-right).     } \label{fig:projpend}
\end{figure}

\begin{figure}[b!]
	\includegraphics[width=.7\textwidth]{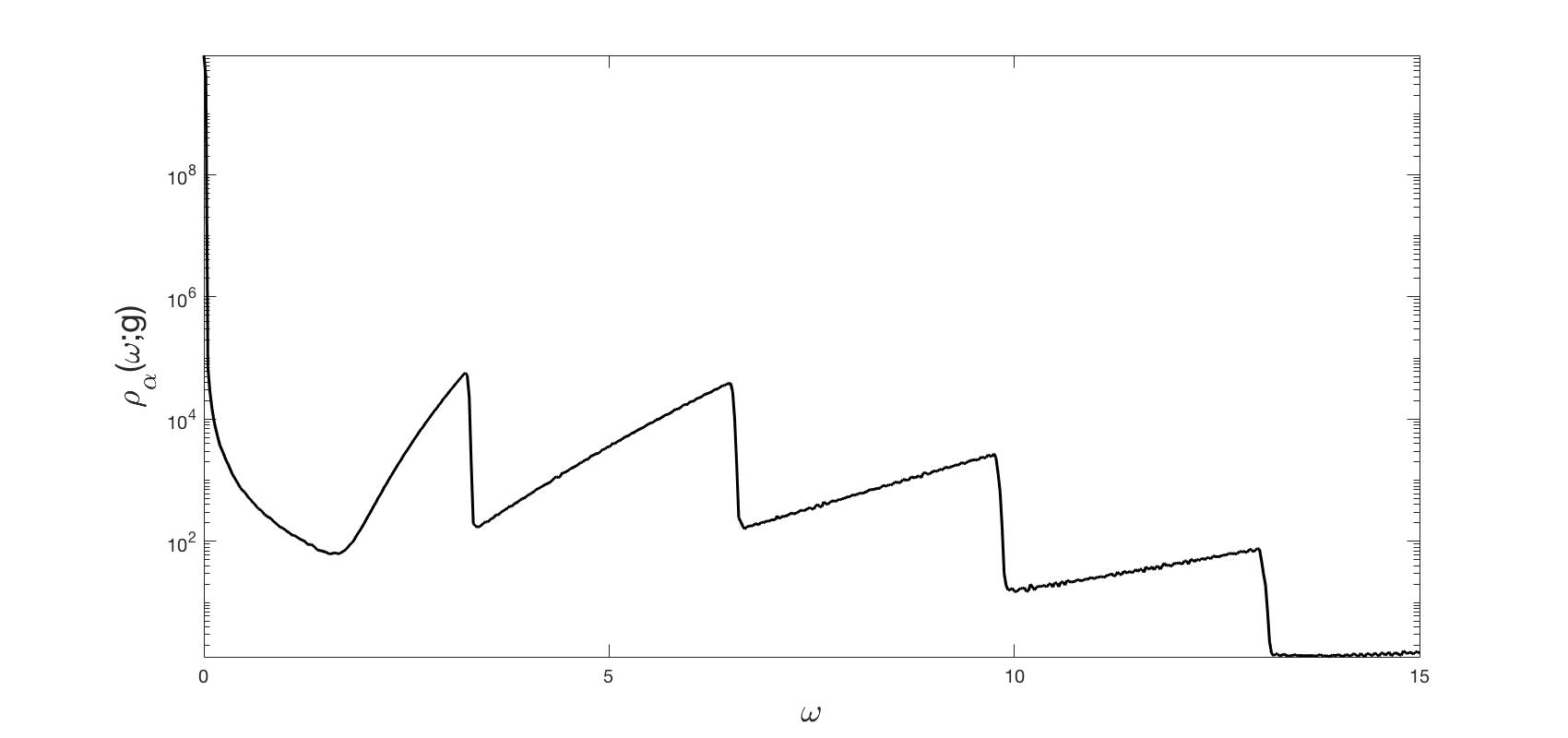}
	\caption{The spectral density function ($\alpha=0.1$) for the observable \eqref{eq:obsduffing1}. } \label{fig:densityduffing}
	\begin{center}
		\begin{tabular}{cc}
			\includegraphics[width=.47\textwidth]{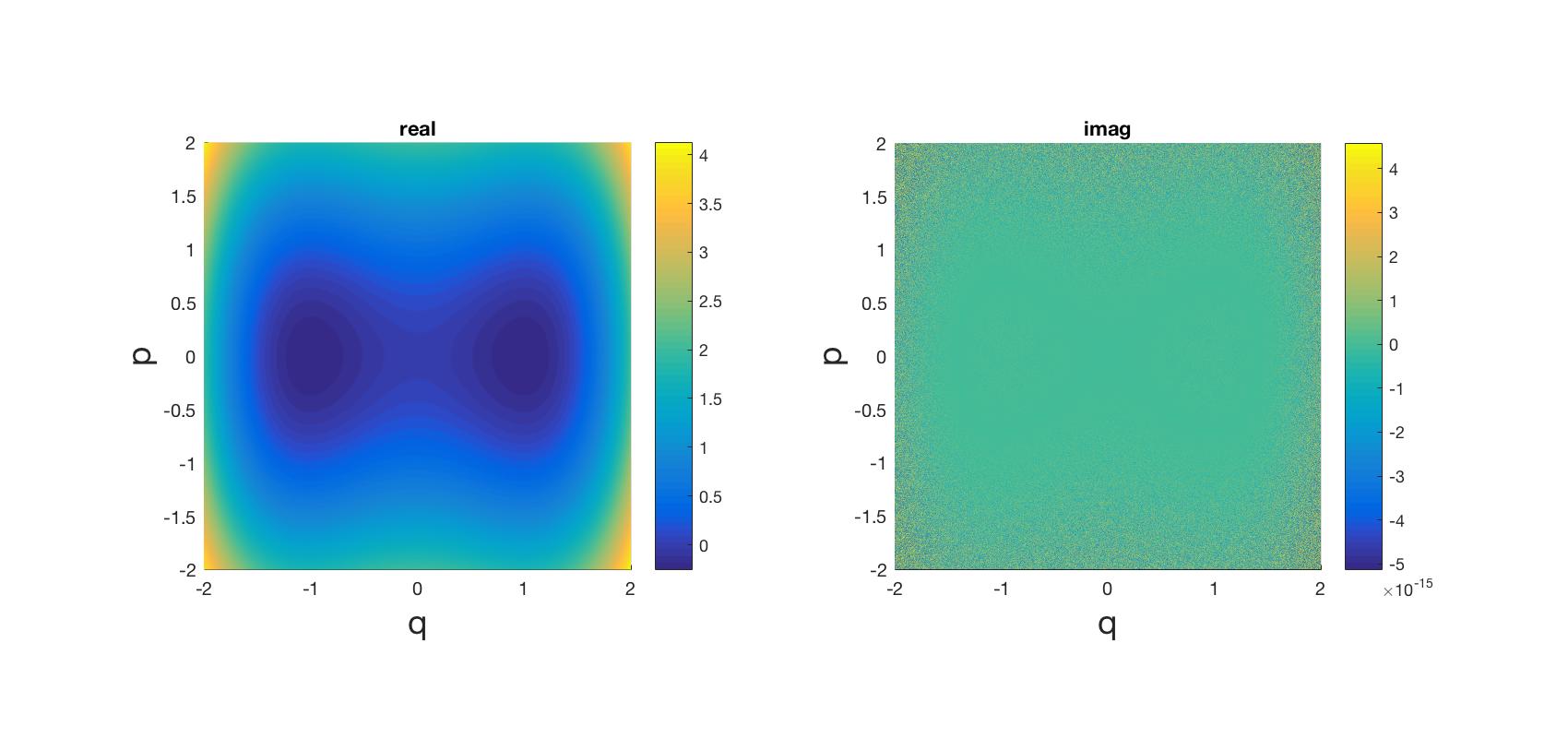} &  \includegraphics[width=.47\textwidth]{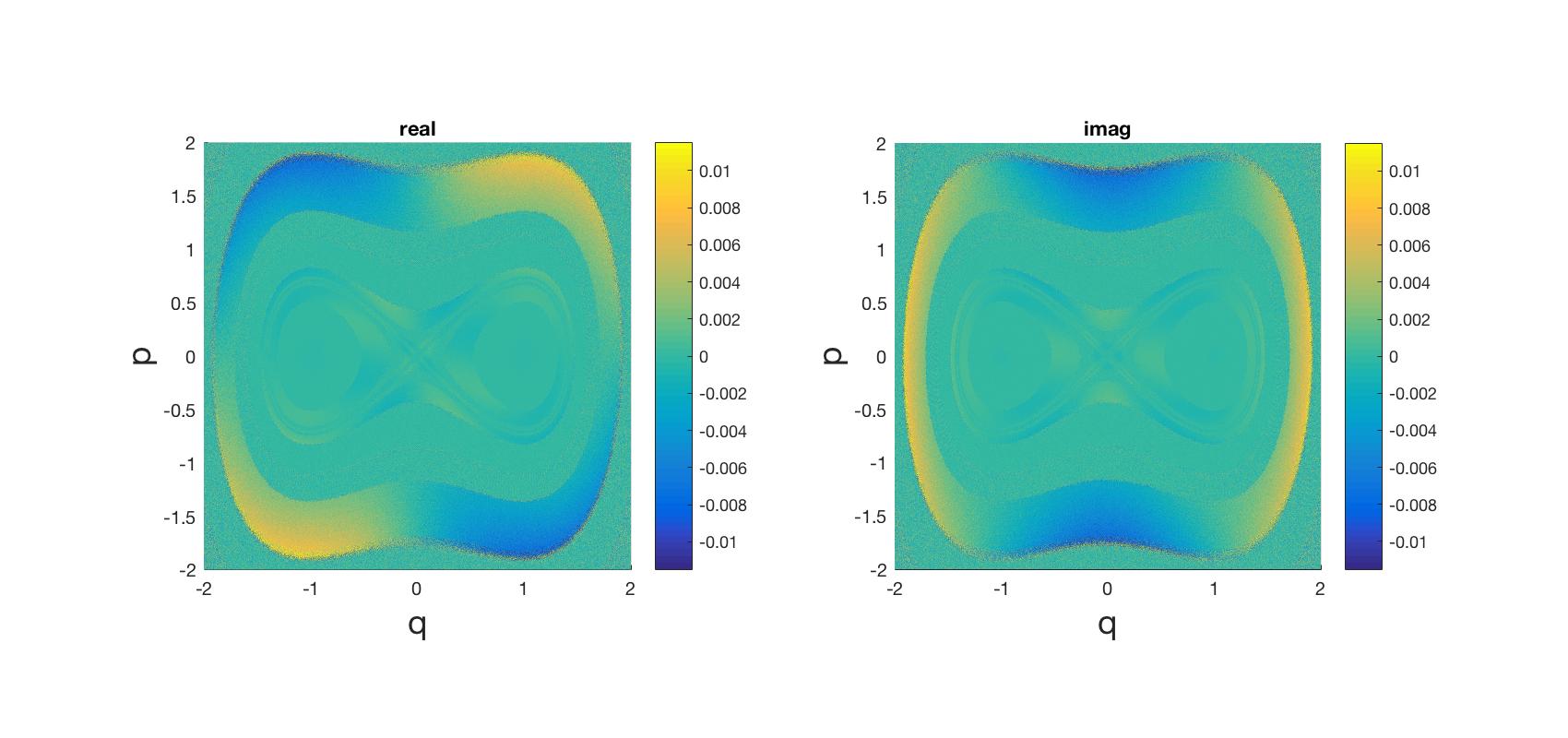} \\
			\includegraphics[width=.47\textwidth]{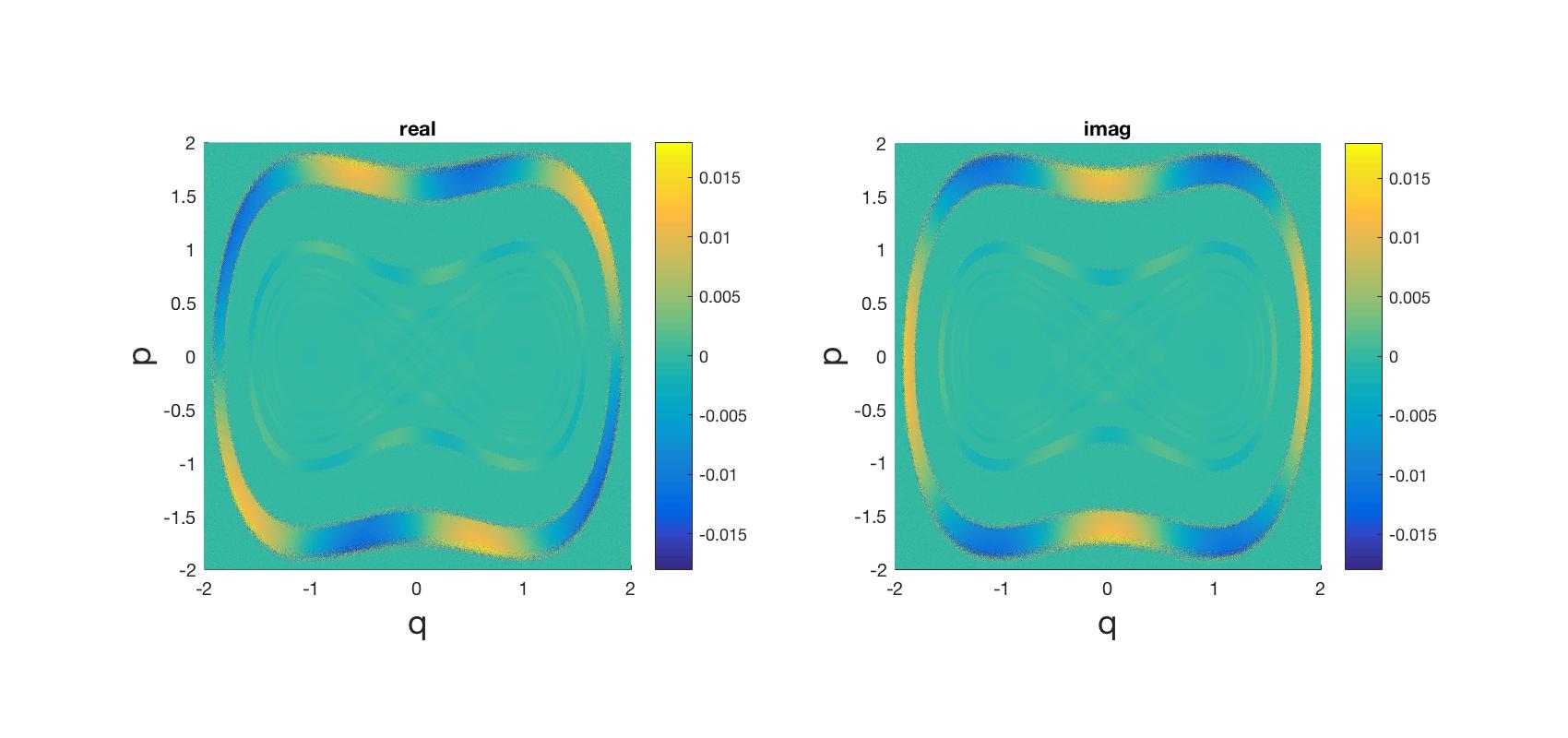} & \includegraphics[width=.47\textwidth]{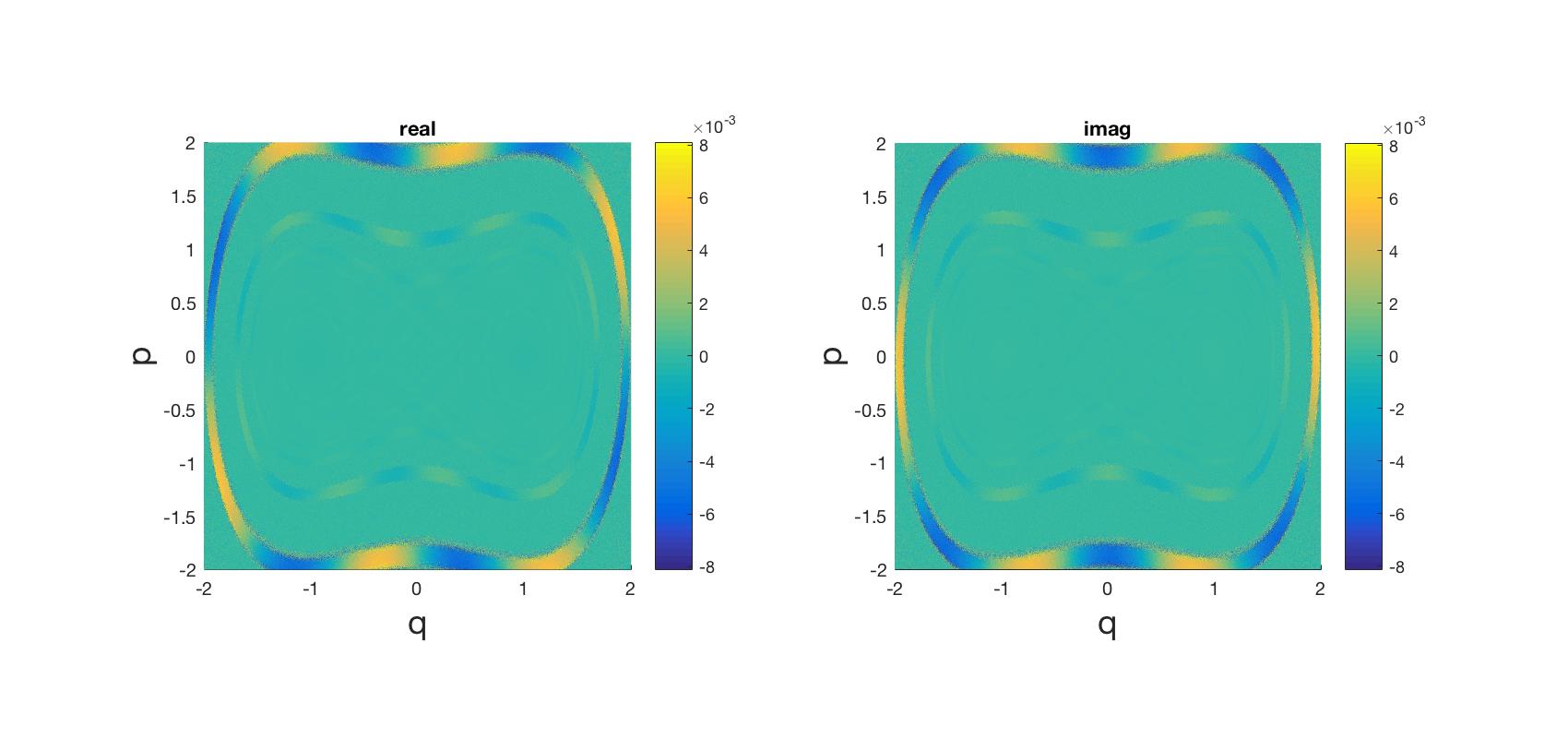} 
		\end{tabular}
	\end{center}
	\caption{Spectral projections for the observable \eqref{eq:obsduffing1} on the invervals $D=[-0.3,0.3)$ (top-left), $D=[2,2.5)$ (top-right), $D=[4.5,5)$ (bottom-left), and $D=[7.5,8.0)$ (bottom-right).     } \label{fig:projduffing}
\end{figure}

\begin{figure}[b!]
	
	\includegraphics[width=.7\textwidth]{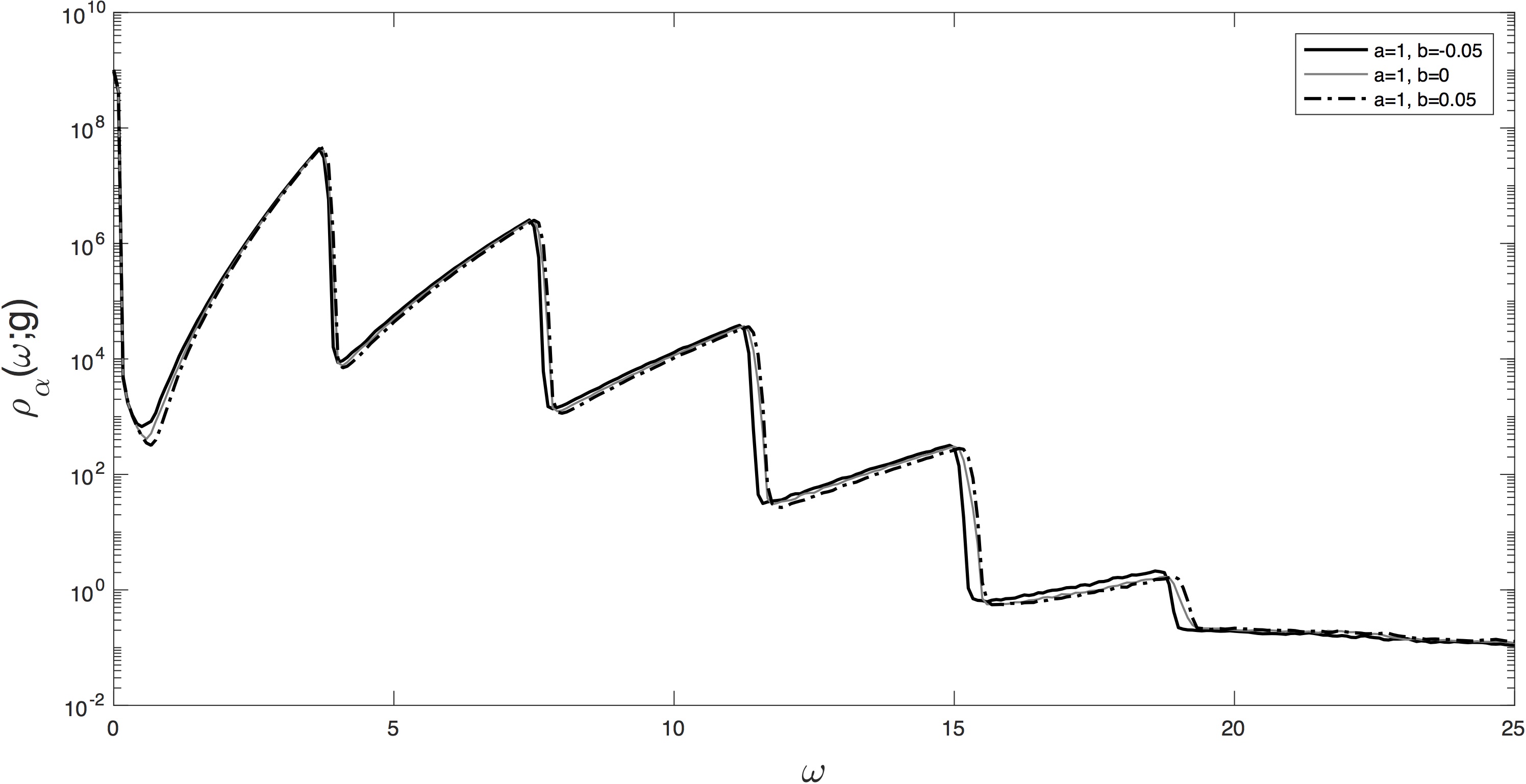}
	\caption{The spectral density function ($\alpha=0.1$) for the observable \eqref{eq:obsduffing}. } \label{fig:spectdensityduffing1}
	
	\includegraphics[width=.7\textwidth]{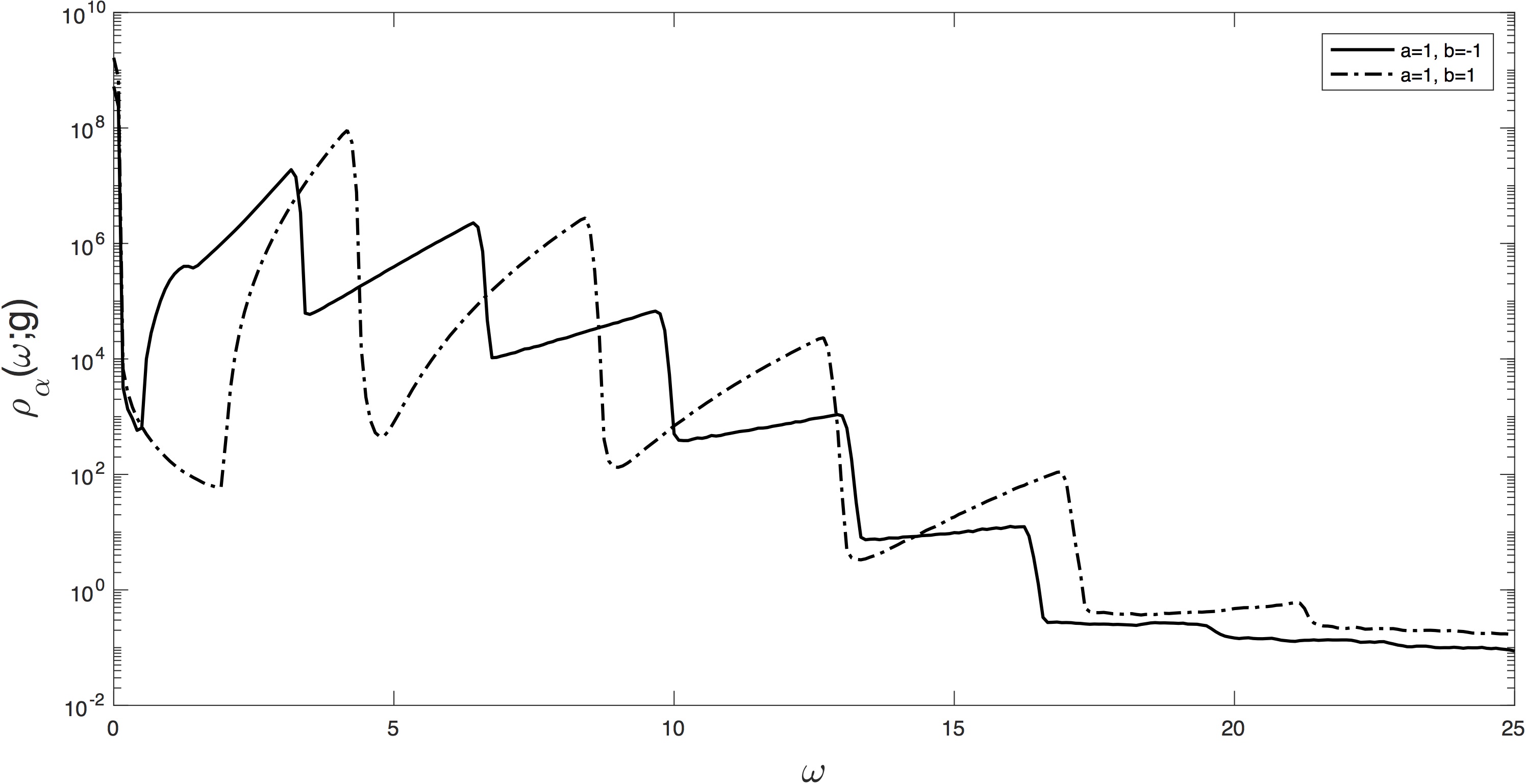} 
	\caption{The spectral density function ($\alpha=0.1$) for the observable \eqref{eq:obsduffing}. } \label{fig:spectdensityduffing2}
\end{figure}

\begin{figure}[b!]
	\begin{center}
		\includegraphics[width=0.7\textwidth]{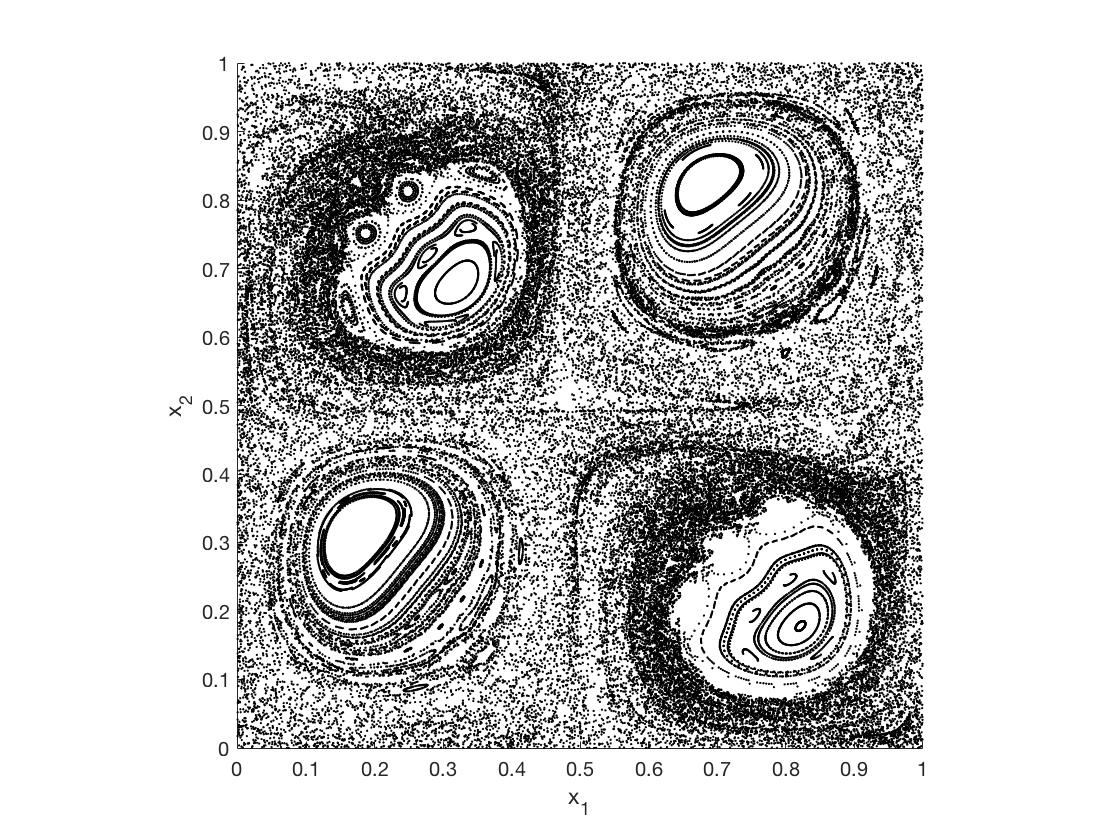} 
	\end{center}
	\caption{Cross-section of trajectories at $x_3=0$ for the quadruple gyre with $A=1/(2\pi)$ and $\epsilon = 0.05$. Clearly noticable in the plot are the KAM tori islands aong with the chaotic region.} \label{fig:poincare}
\end{figure}

\begin{figure}[b!]
	\includegraphics[width=.8\textwidth]{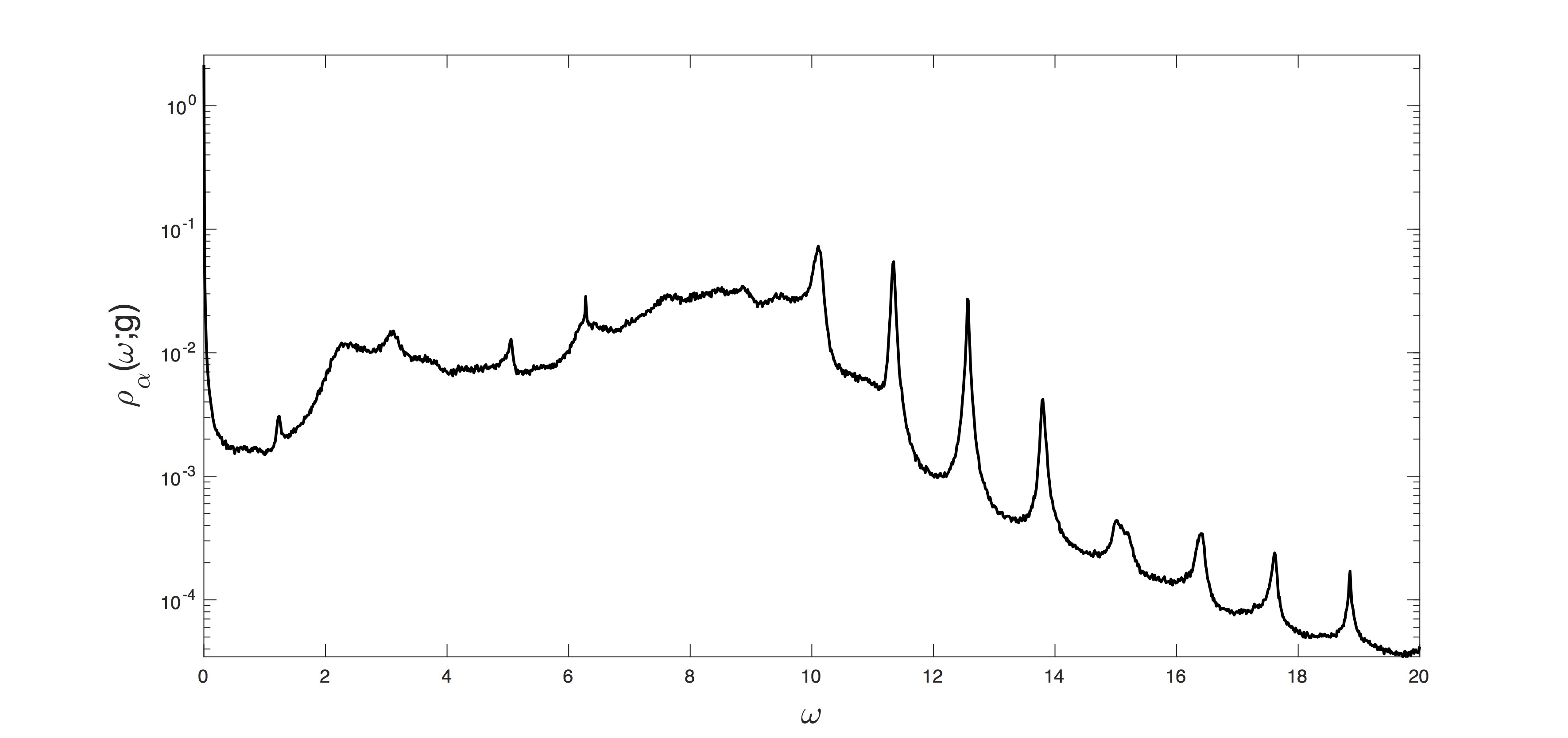} 
	\caption{The spectral density function ($\alpha=0.01$) for the observable \eqref{eq:obsgyre}.} \label{fig:spectdensitygyre}
\end{figure}

\begin{figure}[b!]
	\begin{center}
		\subfigure[$x_3 = 0$]{\includegraphics[width=.49\textwidth]{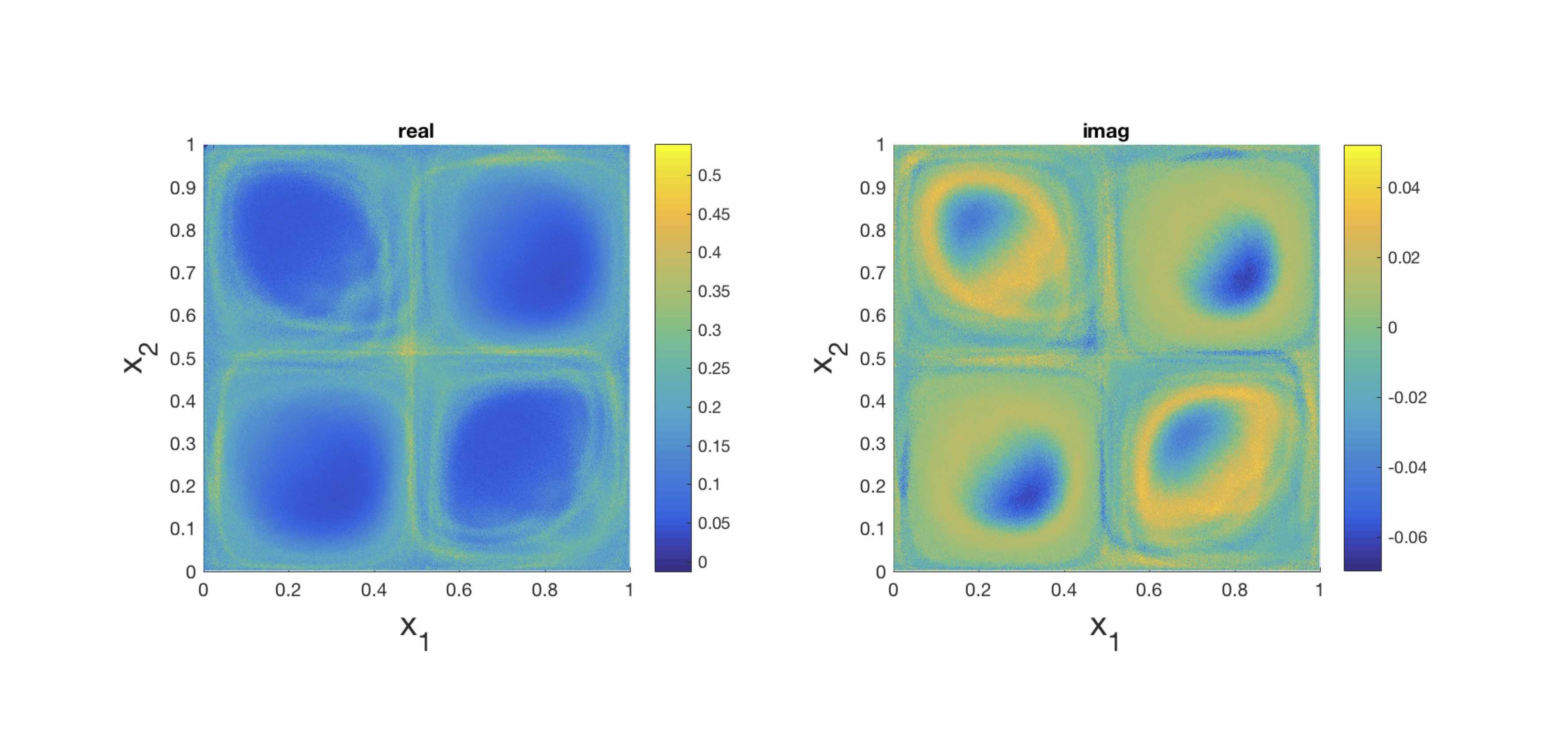} }  \subfigure[$x_3 = 0.25$]{\includegraphics[width=.49\textwidth]{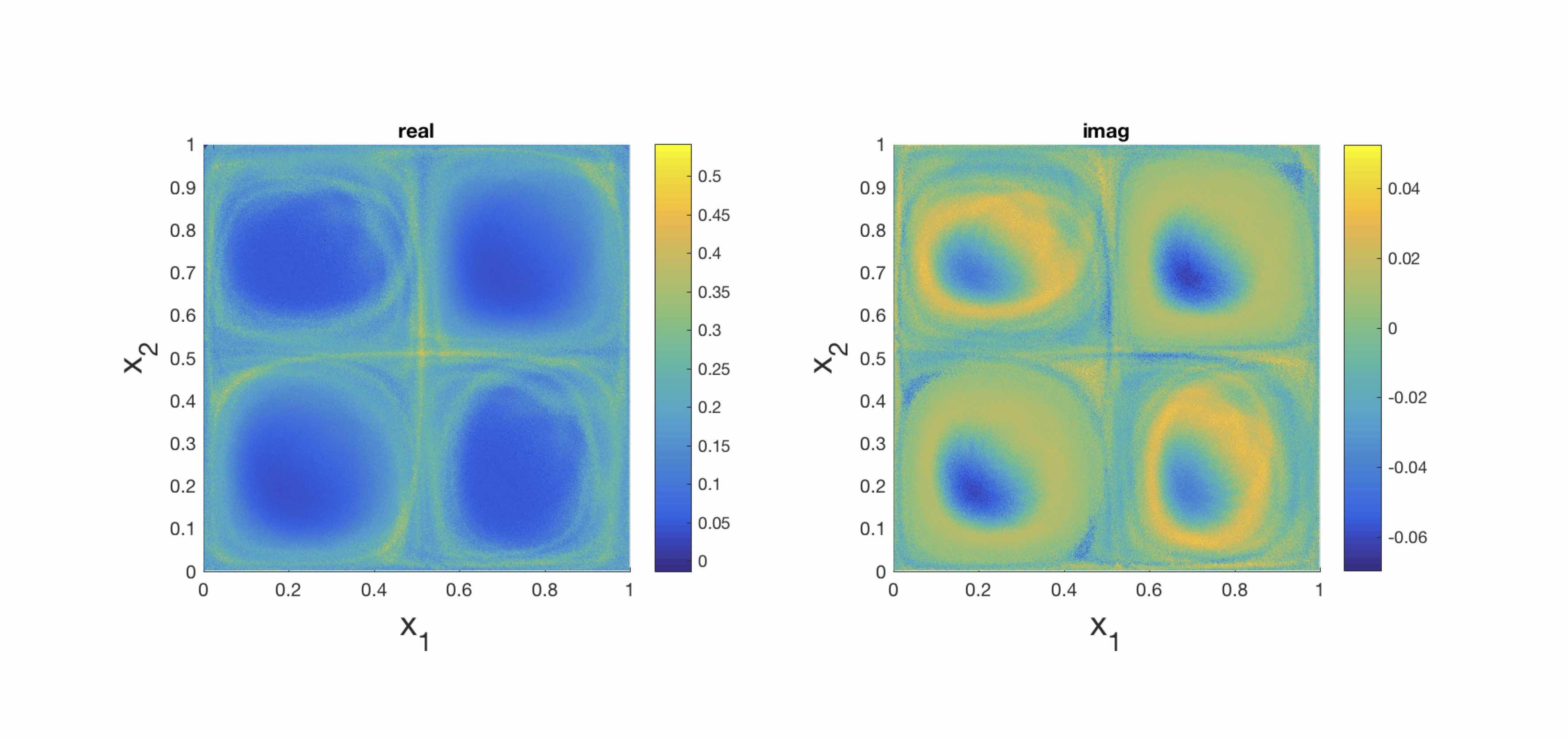} }  \\
		\subfigure[$x_3 = 0.5$]{\includegraphics[width=.49\textwidth]{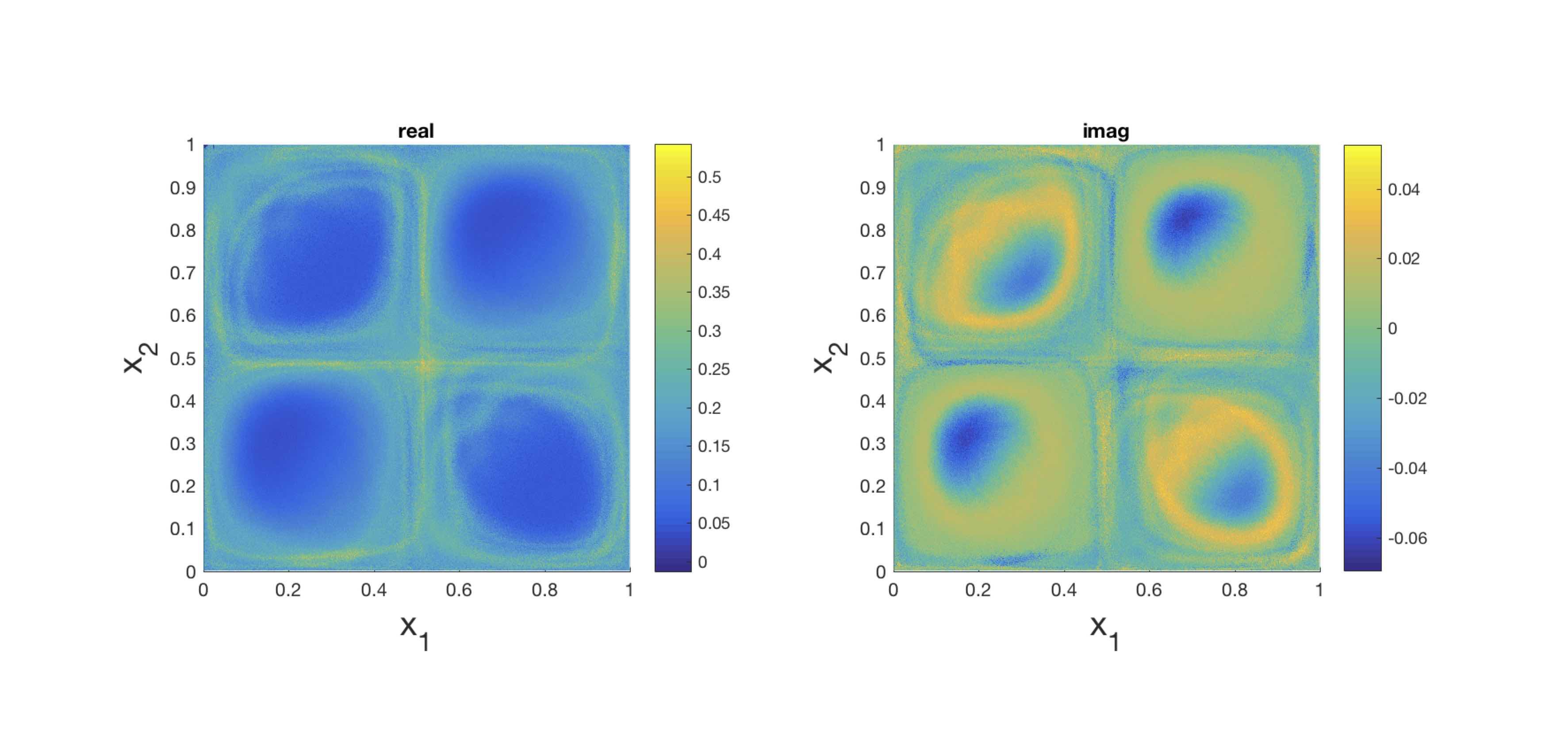} }  
		\subfigure[$x_3 = 0.75$]{\includegraphics[width=.49\textwidth]{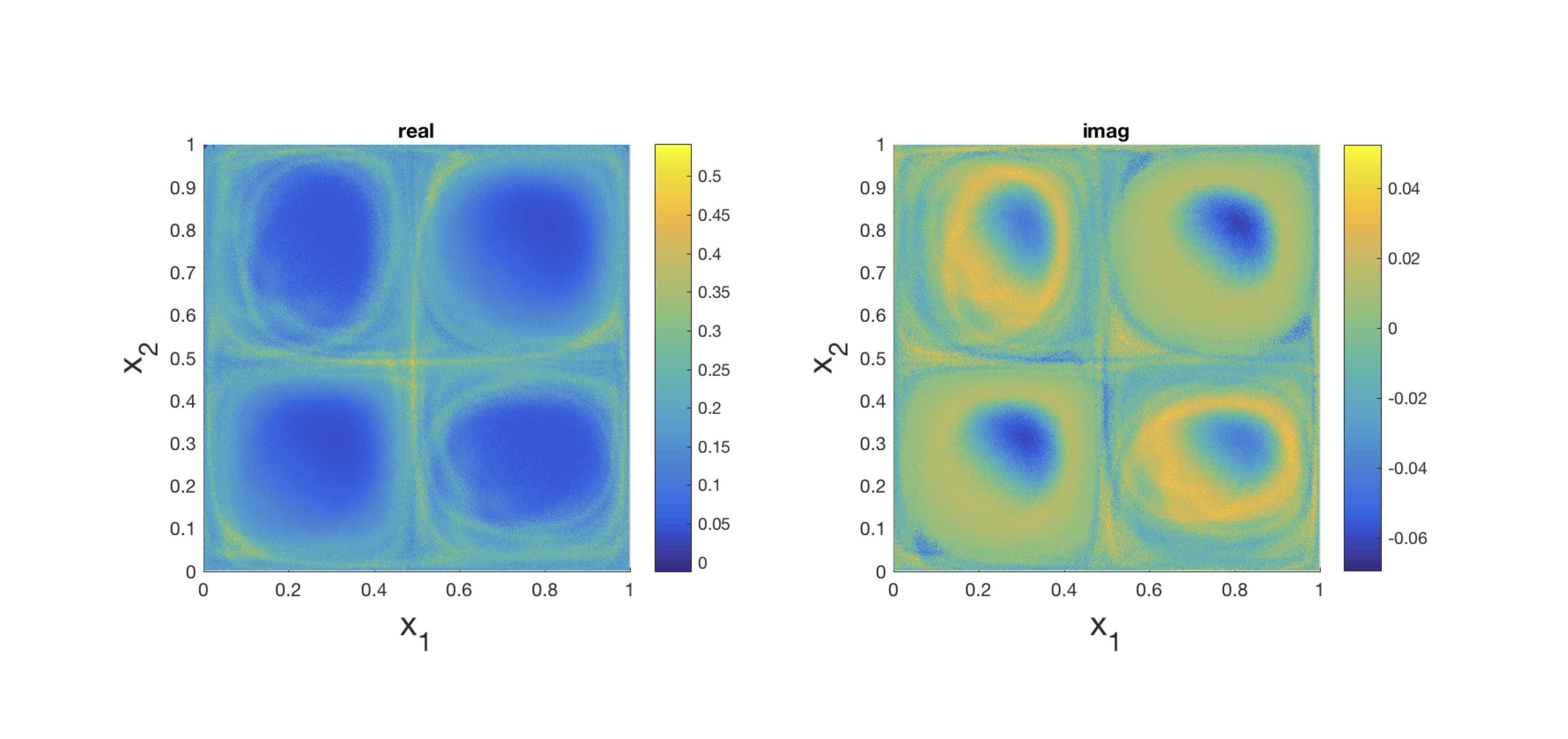} } 
	\end{center}
	\caption{Spectral projection of the observable \eqref{eq:obsgyre} onto the interval $D=[-0.4,0.4)$. } \label{fig:gyre1}
\end{figure}

\begin{figure}[b!]
	\begin{center}
		\subfigure[$x_3 = 0$]{\includegraphics[width=.49\textwidth]{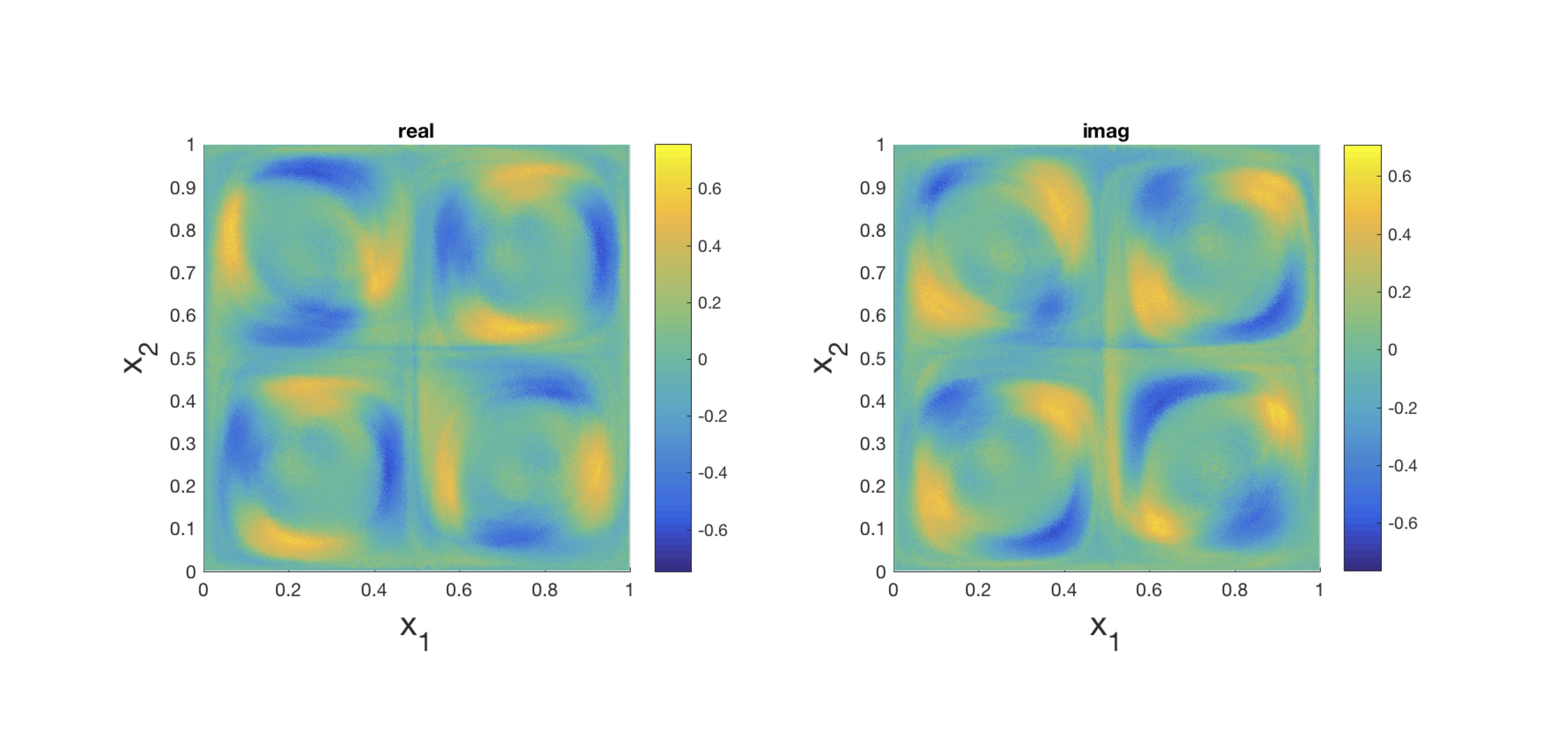} }  \subfigure[$x_3 = 0.25$]{\includegraphics[width=.49\textwidth]{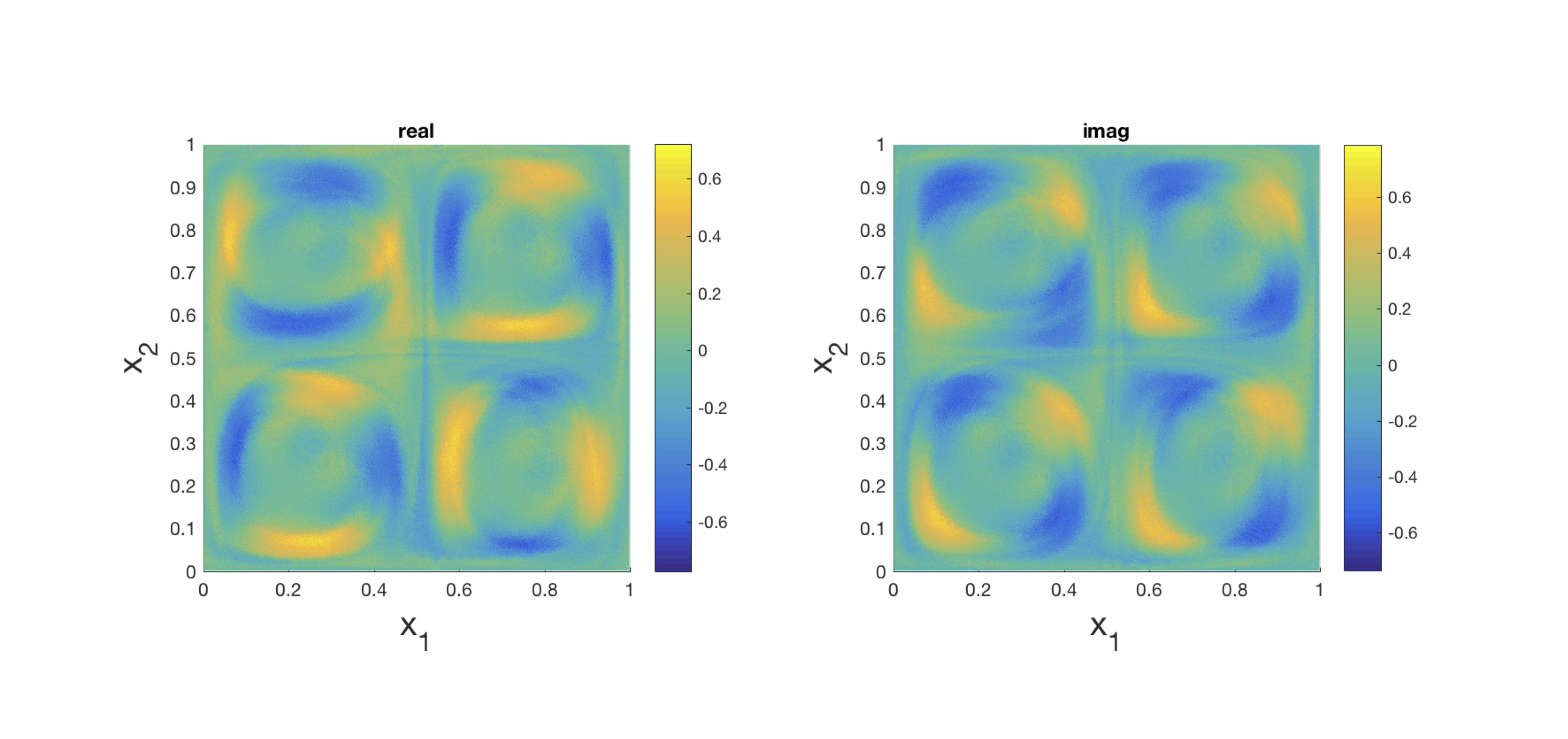} }  \\
		\subfigure[$x_3 = 0.5$]{\includegraphics[width=.49\textwidth]{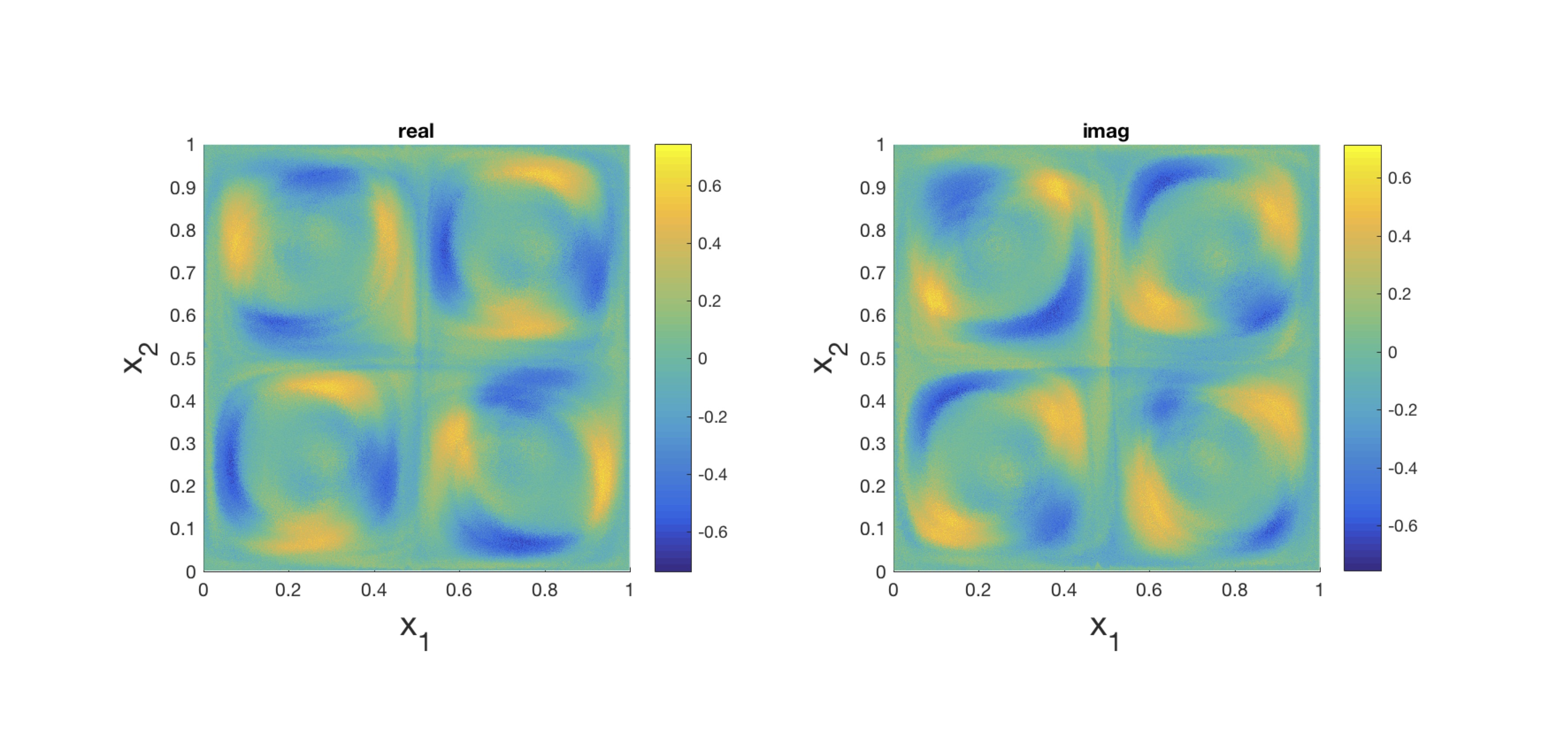} }  
		\subfigure[$x_3 = 0.75$]{\includegraphics[width=.49\textwidth]{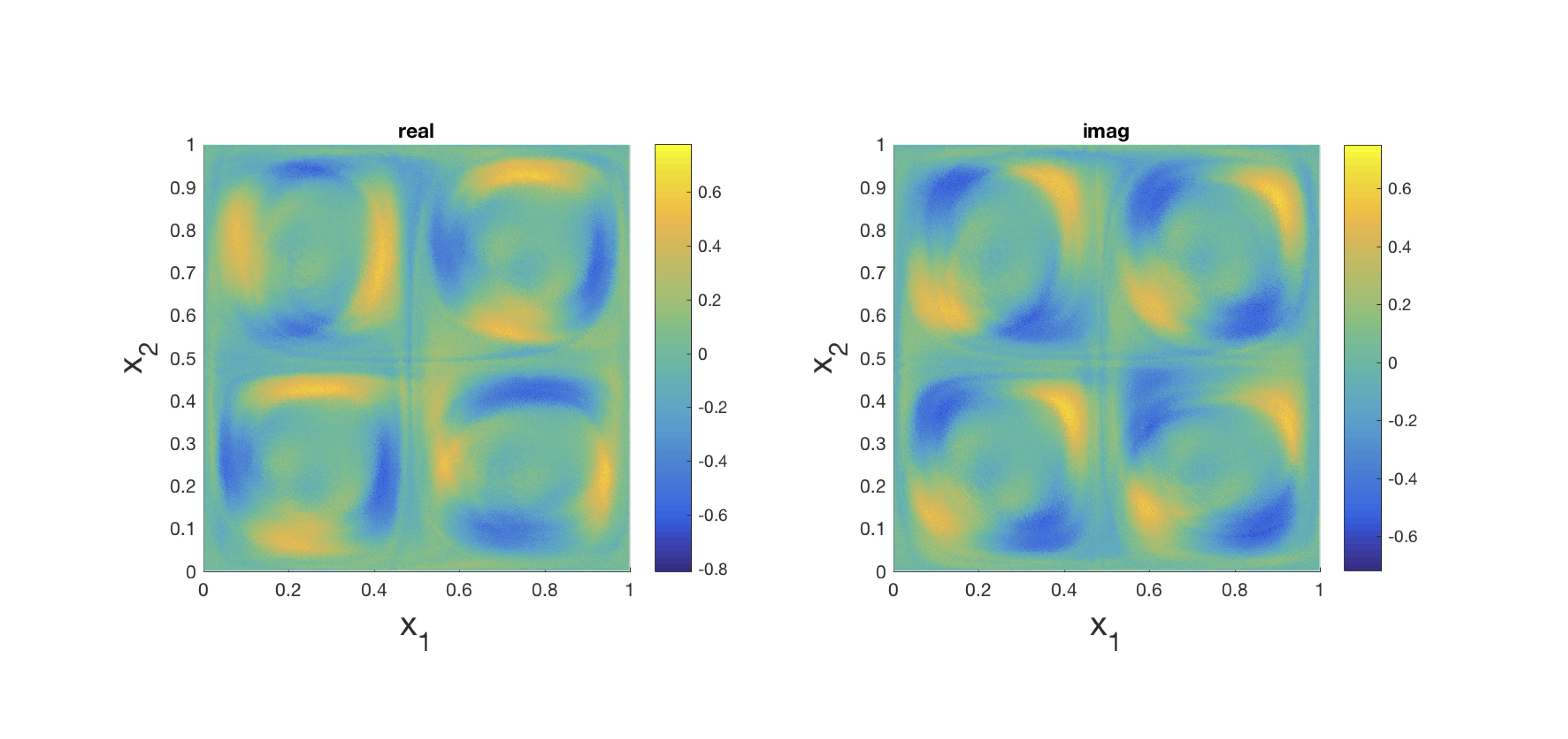} } 
	\end{center}
	\caption{Spectral projection of the observable \eqref{eq:obsgyre} onto the interval $D=[6.5,10.0)$. } \label{fig:gyre2}
\end{figure}

\begin{figure}[b!]
	\begin{center}
		\subfigure[$x_3 = 0$]{\includegraphics[width=.49\textwidth]{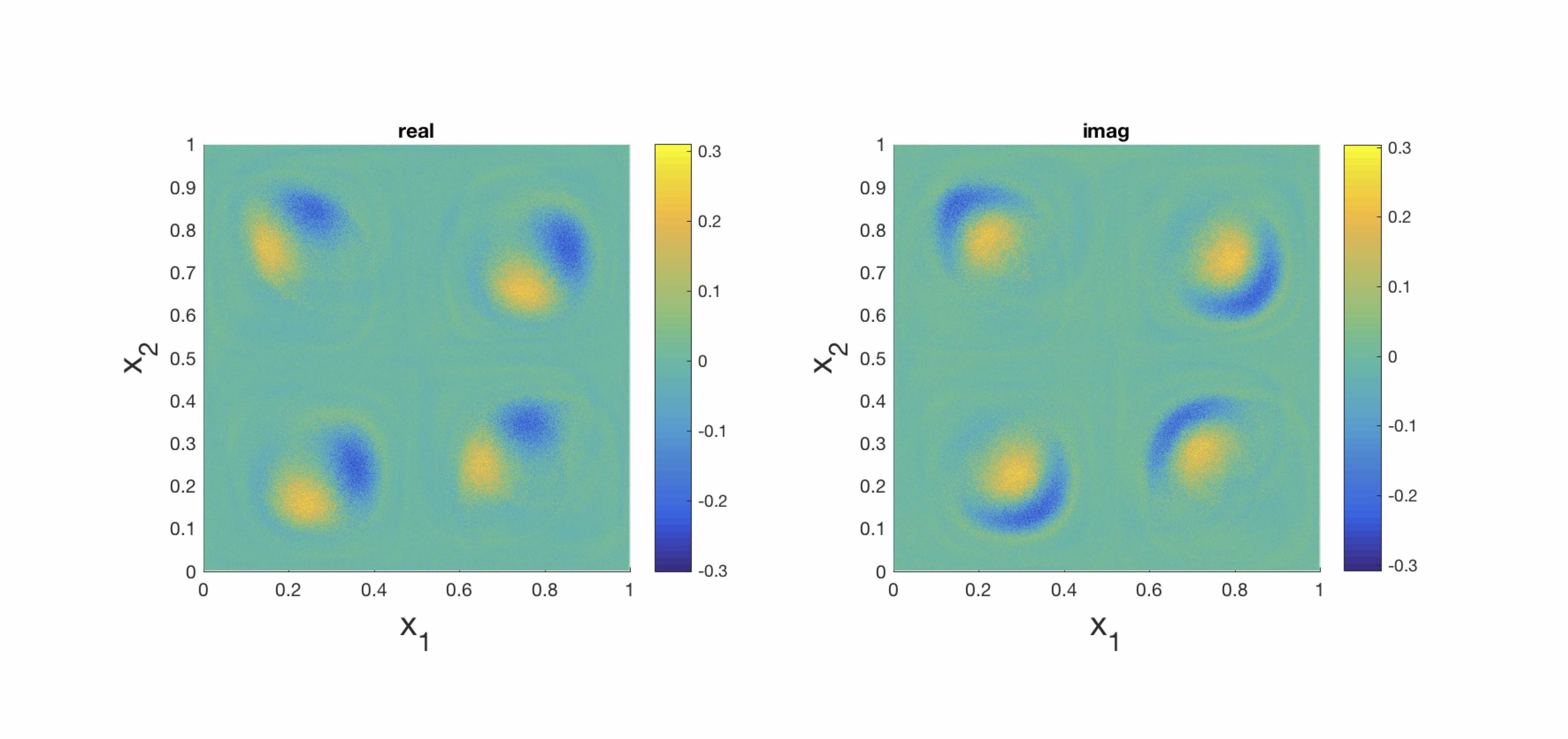} }  \subfigure[$x_3 = 0.25$]{\includegraphics[width=.49\textwidth]{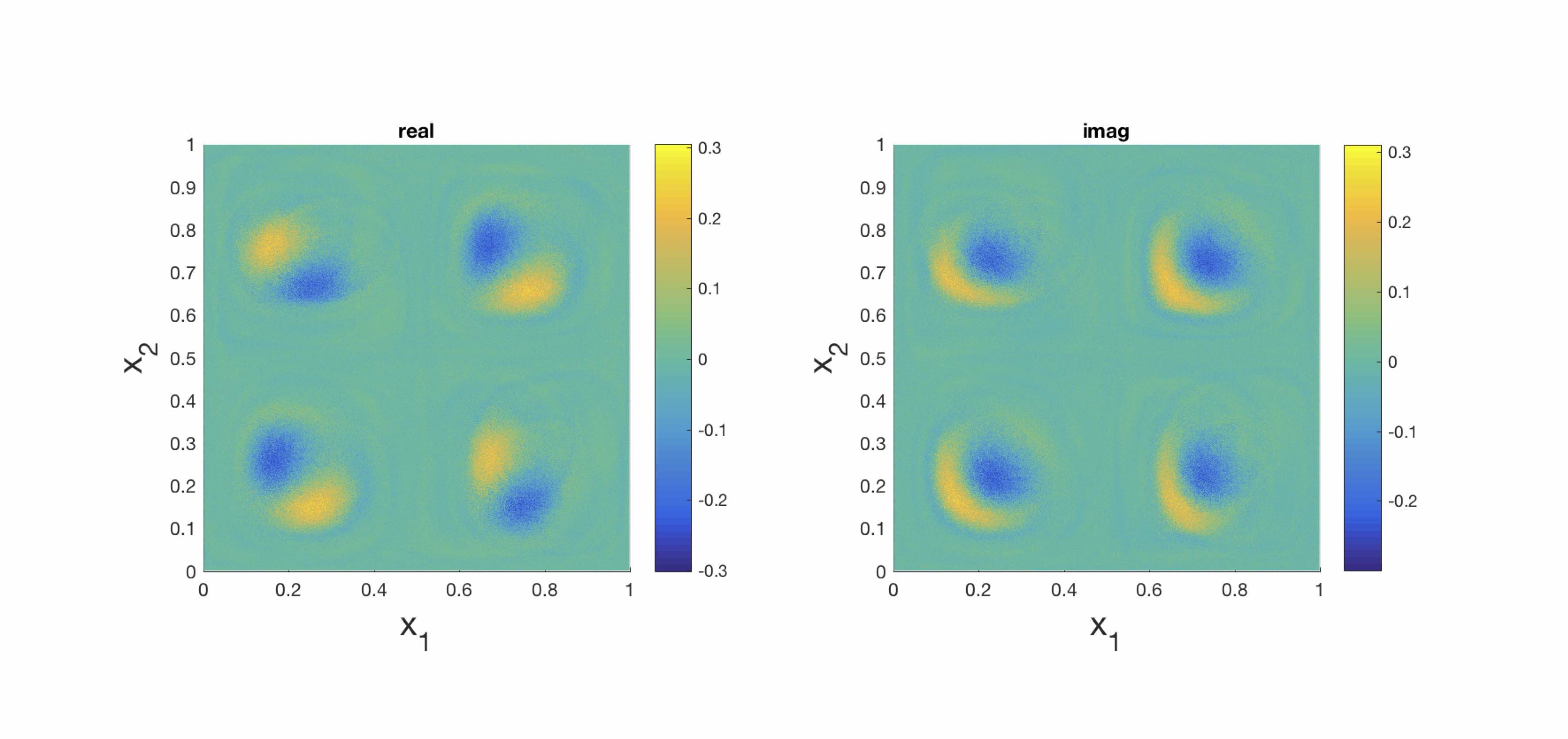} }  \\
		\subfigure[$x_3 = 0.5$]{\includegraphics[width=.49\textwidth]{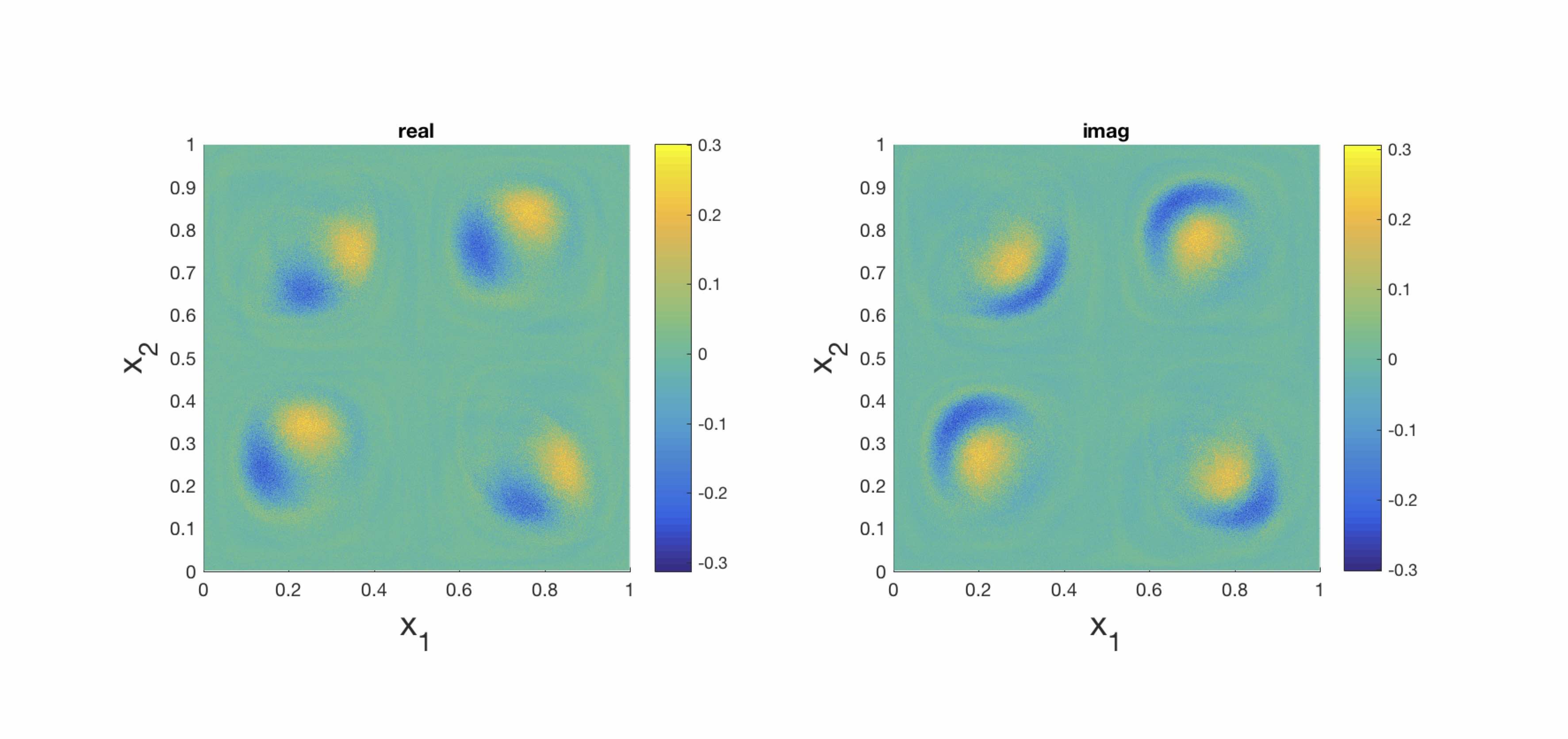} }  
		\subfigure[$x_3 = 0.75$]{\includegraphics[width=.49\textwidth]{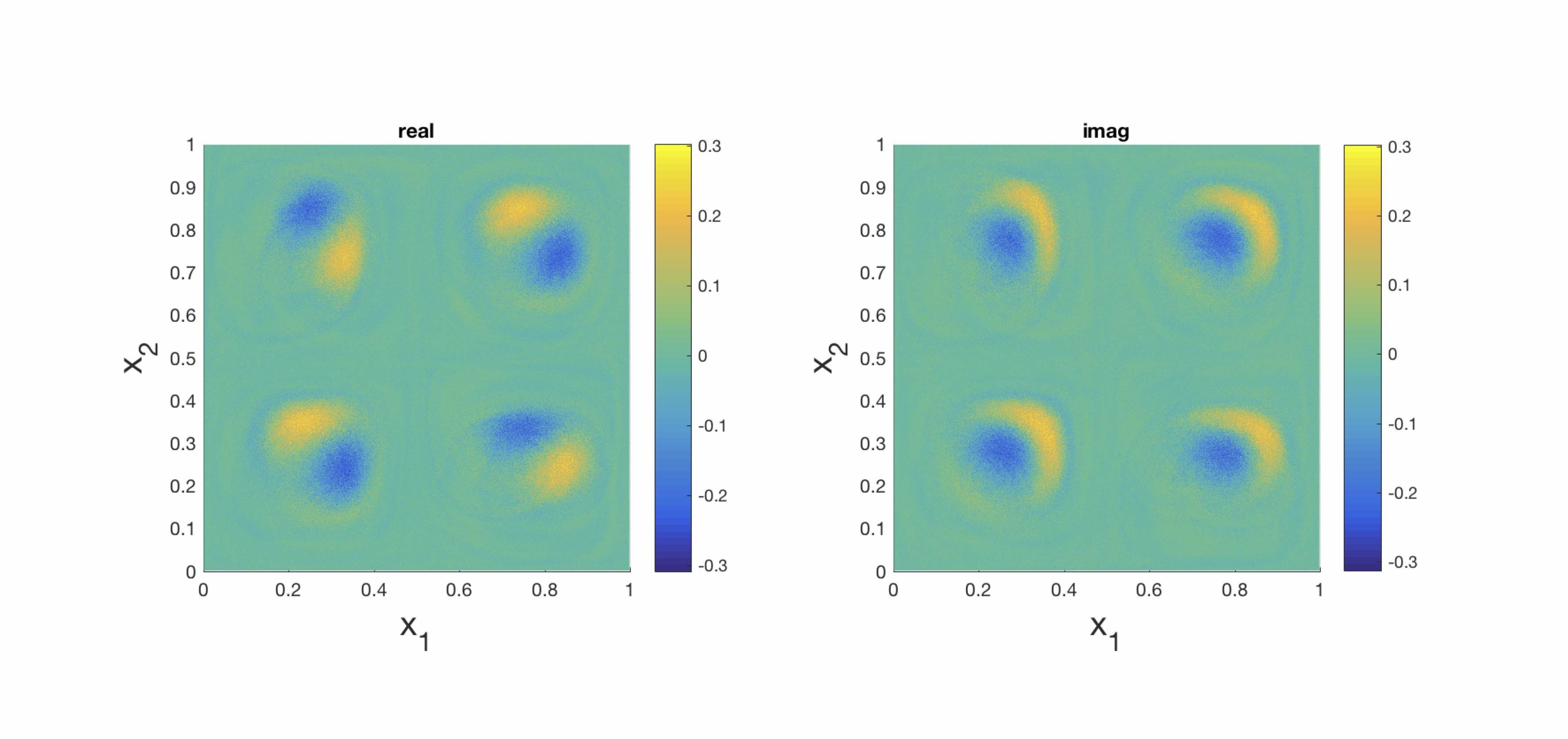} } 
	\end{center}
	\caption{Spectral projection of the observable \eqref{eq:obsgyre} onto the interval $D=[-11.1,11.4)$. } \label{fig:gyre3}
\end{figure}

\begin{figure}[b!]
	\includegraphics[width=.8\textwidth]{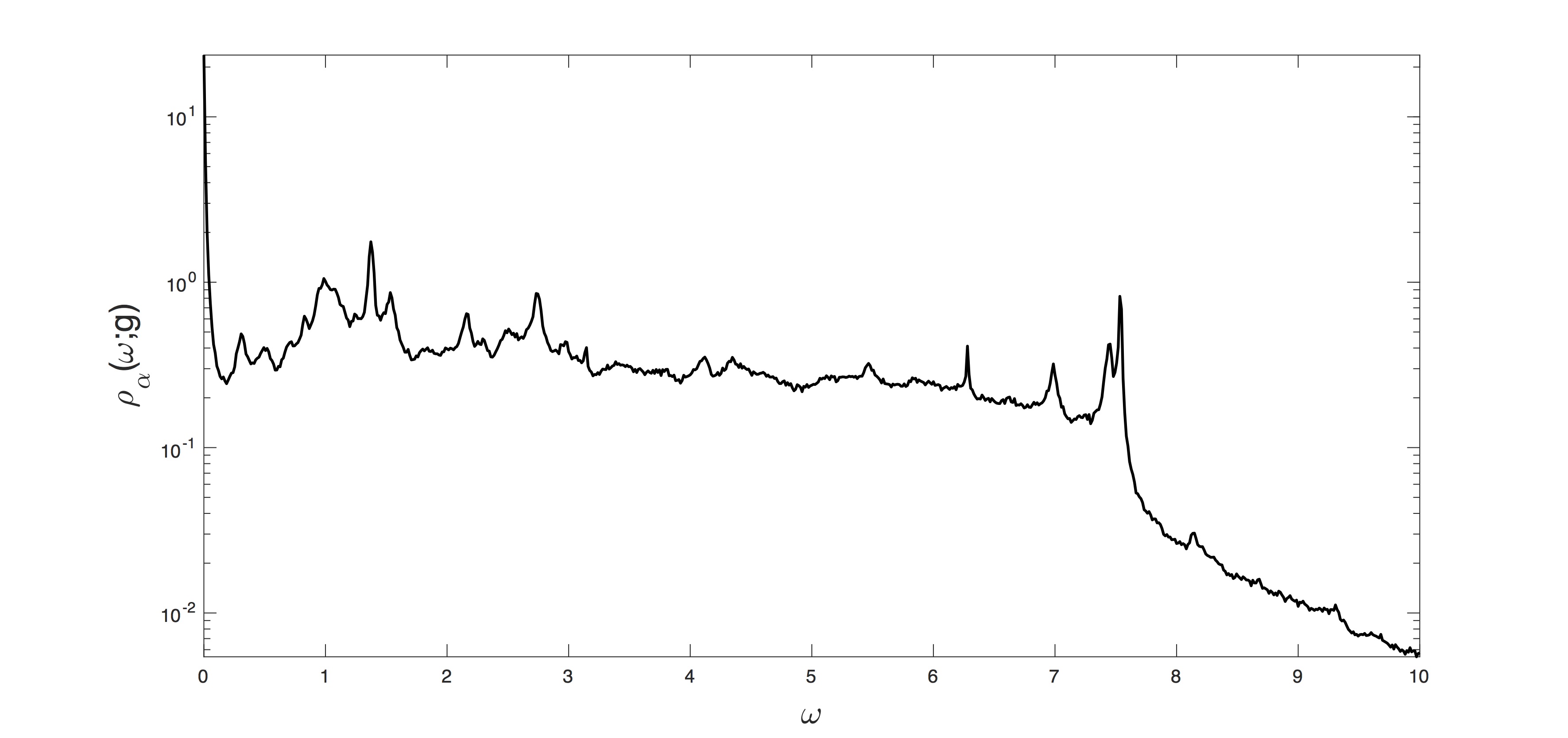} 
	\caption{The spectral density function ($\alpha=0.01$) for the observable \eqref{eq:obsABC}. } \label{fig:spectdensityABC}
\end{figure}

\begin{figure}[b!]
	\includegraphics[width=.8\textwidth]{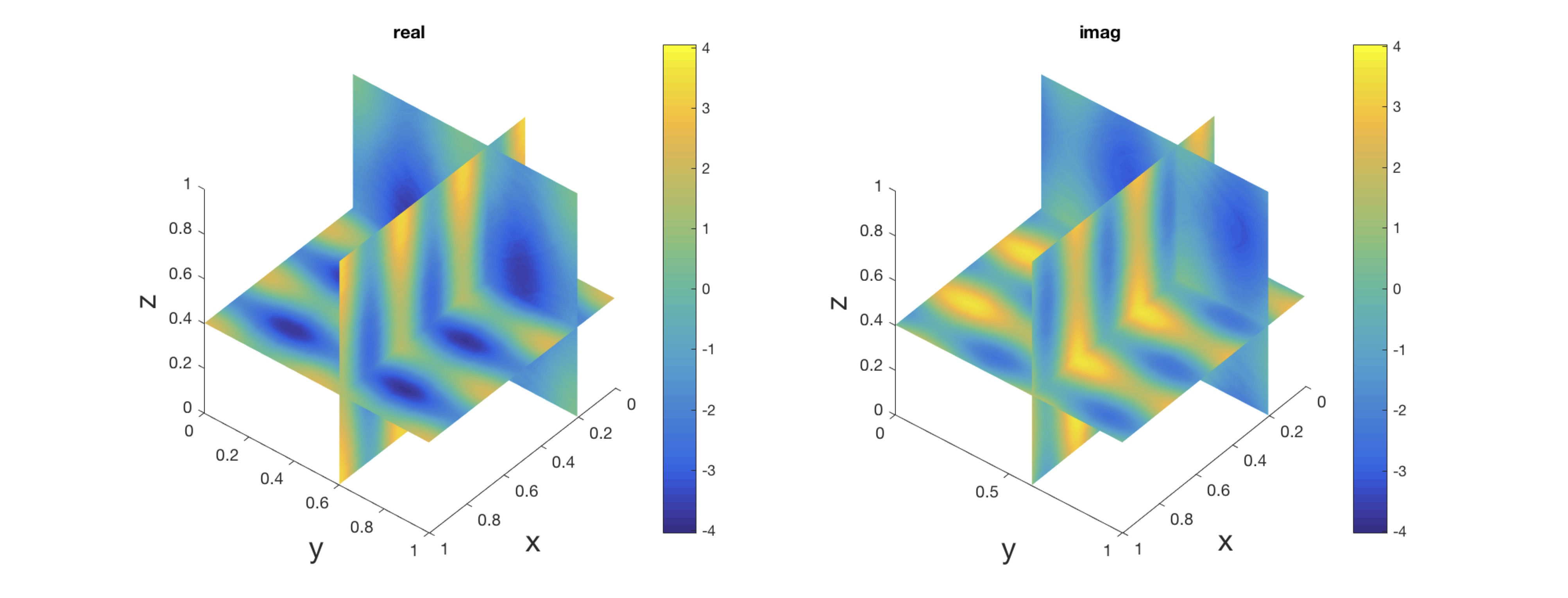} 
	\caption{Spectral projection of the observable \eqref{eq:obsABC} onto the interval $D=[-0.3,0.3)$.} \label{fig:projABC1}
	\includegraphics[width=.8\textwidth]{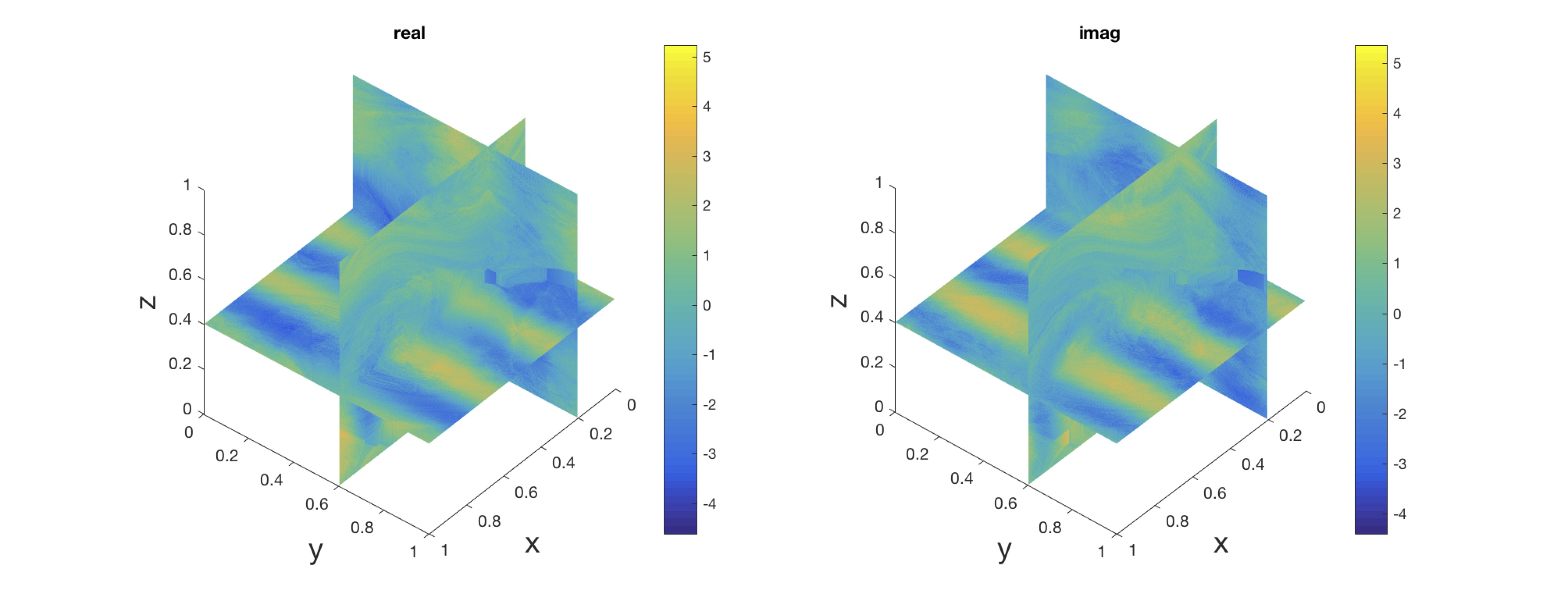} 
	\caption{Spectral projection of the observable \eqref{eq:obsABC} onto the interval $D=[7.36,7.56)$.} \label{fig:projABC2}
\end{figure}

\end{document}